\documentclass[10pt]{article}

\usepackage[utf8]{inputenc}
\usepackage{amsmath,amsfonts,amssymb}
\usepackage{amsthm}
\usepackage[a4paper]{geometry}
\usepackage{graphicx}       
\usepackage{mathtools}
\usepackage{subcaption}
\usepackage{bbm}
\usepackage{xcolor}
\usepackage{dsfont}
\usepackage{enumitem}
\usepackage{overpic}
\usepackage{orcidlink}
\hypersetup{hidelinks}

\newcommand{\R}{{\mathbb R}}

\newcommand{\N}{{\mathbb N}}

\newcommand*{\dd}{\mathrm{d}}

\def\R{\mathbb{R}}

\def\N{\mathbb{N}}

\def\cE{\mathcal{E}}

\def\cG{\mathcal{G}}

\def\cI{\mathcal{I}}

\def\cL{\mathcal{L}}
\def\cM{\mathcal{M}}

\def\cT{\mathcal{T}}

\def\cV{\mathcal{V}}

\def\cX{\mathcal{X}}

\def\txtd{{\textnormal{d}}}

\def\txtD{{\textnormal{D}}}

\def\ra{\rightarrow}
\def\I{\infty}

\newcommand{\be}{\begin{equation}}
	\newcommand{\ee}{\end{equation}}
\newcommand{\benn}{\begin{equation*}}
	\newcommand{\eenn}{\end{equation*}}
\newcommand{\bea}{\begin{eqnarray}}
	\newcommand{\eea}{\end{eqnarray}}
\newcommand{\beann}{\begin{eqnarray*}}
	\newcommand{\eeann}{\end{eqnarray*}}

\newcommand{\myendex}{$\blacklozenge$\end{ex}}
\newcommand{\myendexerc}{$\lozenge$\end{exerc}}
\newcommand{\myendpexerc}{$\lozenge$\end{pexerc}}

\newtheorem{theorem}{Theorem}[section]

\newtheorem{definition}[theorem]{Definition}

\newtheorem{proposition}[theorem]{Proposition}

\newtheorem{remark}[theorem]{Remark}

\begin{document}

\title{\textbf{A Dynamical Systems Perspective \\ on the Analysis of Neural Networks}}
\vspace{2mm}
\author{Dennis Chemnitz \orcidlink{0000-0002-3303-3533}$^1$, Maximilian Engel \orcidlink{0000-0002-1406-8052}$^{1,2}$, \\[1mm] Christian Kuehn \orcidlink{0000-0002-7063-6173}$^{3,4,5}$ \& Sara-Viola Kuntz \orcidlink{0009-0000-4611-9742}$^{3,4,5}$}

\date{
	\small{$^1$\textit{Freie Universität Berlin, Fachbereich Mathematik, Arnimallee 6, 14195 Berlin, Germany} \\
	$^2$\textit{Universiteit van Amsterdam, Korteweg de-Vries Institute for Mathematics, Science Park 105-107, \\ 1098 XG Amsterdam, The Netherlands} \\
	$^3$\textit{Technical University of Munich, School of Computation, Information and Technology, \\ Department of Mathematics, Boltzmannstraße 3, 85748 Garching, Germany} \\
	$^4$\textit{Munich Data Science Institute (MDSI), Garching, Germany } \\
	$^5$\textit{Munich Center for Machine Learning (MCML), München, Germany }}\\[7mm]
	\large{June 16, 2026}
}

\vspace{2mm}

\maketitle

\begin{abstract}
	In this chapter, we utilize dynamical systems to analyze several aspects of machine learning algorithms. As an expository contribution we demonstrate how to re-formulate a wide variety of challenges from deep neural networks, (stochastic) gradient descent, and related topics into dynamical statements. We also tackle three concrete challenges. First, we consider the process of information propagation through a neural network, i.e., we study the input-output map for different architectures. We explain the universal embedding property for augmented neural ODEs representing arbitrary functions of given regularity, the classification of multilayer perceptrons and neural ODEs in terms of suitable function classes, and the memory-dependence in neural delay equations. Second, we consider the training aspect of neural networks dynamically. We describe a dynamical systems perspective on gradient descent and study stability for overdetermined problems. We then extend this analysis to the overparameterized setting and describe the edge of stability phenomenon, also in the context of possible explanations for implicit bias. For stochastic gradient descent, we present stability results for the overparameterized setting via Lyapunov exponents of interpolation solutions. Third, we explain several results regarding mean-field limits of neural networks. We describe a result that extends existing techniques to heterogeneous neural networks involving graph limits via digraph measures. This shows how large classes of neural networks naturally fall within the framework of Kuramoto-type models on graphs and their large-graph limits. Finally, we point out that similar strategies to use dynamics to study explainable and reliable AI can also be applied to settings such as generative models or fundamental issues in gradient training methods, such as backpropagation or vanishing/exploding gradients.
\end{abstract}

\section{Introduction}
\label{cekk_sec:intro} 

In light of the overwhelming success of artificial intelligence in recent years and its inevitable impact on many aspects of our society, a concerning disparity has emerged between the practical application of machine learning and our theoretical understanding of it. Due to the dynamic nature of both the tasks and algorithms involved, machine learning (ML) and artificial intelligence (AI) are deeply linked to the well-established mathematical theory of dynamical systems. This rich interplay opens up many possibilities for theoretical analysis and understanding.
Broadly, two main directions have emerged: (a) we use ML/AI techniques to better understand a given dynamical system, and (b) we view ML algorithms as dynamical systems themselves and try to understand ML/AI abilities using dynamics techniques. Both directions have been active for many decades~\cite{FarmerPackardPerelson,Kosko} but have recently re-gained significant interest~\cite{BruntonKutz,Weinan2017}. There have been many interesting avenues regarding (a), which are often somewhat more direct to explore, since ML can be often be used as an almost black-box tool to interpolate or extrapolate dynamics from data. In this book chapter, we entirely focus on the second direction (b). This viewpoint tends to be very challenging and will likely remain so for the foreseeable future. Some of the challenges are:

\begin{itemize}[leftmargin=9mm]
	\item[(C1)] When is it possible to actually reformulate a problem about explainable/reliable AI in a dynamics framework?
	\item[(C2)] In a dynamics context, which additional challenges do ML algorithms pose in their analysis in comparison to other nonlinear systems?
	\item[(C3)] How many analytical results can one expect to justify completely via rigorous proofs?
\end{itemize}

We shall answer (C1) in this chapter affirmatively by pointing out several instances of concretely reformulating questions about AI/ML. Reinterpreting virtually any ML algorithm dynamically can provide a very effective angle of attack for mathematical/theoretical analysis. This reformulation is entirely natural, e.g., any artificial neural network processes information dynamically as a dynamical system \emph{on} a network, while training the weights is usually done iteratively, leading to a dynamical system \emph{of} a network. This viewpoint of neural networks can also be well-embedded into the framework of adaptive (or co-evolutionary) networks that combine both dynamical aspects~\cite{Berneretal}. 
In fact, many classical results in dynamical systems can be immediately useful for understanding dynamics of ML algorithms~\cite{https://doi.org/10.48550/arxiv.2110.10295,Sharkovskii1}.
Yet, a key aspect of (C2) is that systems are usually so high-dimensional and heterogeneous in practice that preliminary reductions of the dynamics are necessary. We shall follow this paradigm here, trying to first separate the questions of neural network design (i.e., dynamics on the network) from training the weights of a given architecture (i.e., dynamics of the network). Furthermore, we are going to point out that many challenges to explain ML/AI can usually be simplified significantly, when one isolates just a single aspect of the algorithm. Only after a suitable reduction, we can hope for a full analytical proof regarding challenge (C3).

A key theme of this chapter is that, throughout, we emphasize and explain why re-considering AI/ML as a dynamical problem is natural and advantageous. Instead of then tackling the challenges~(C2) and (C3) in full generality, we focus on illustrating them via results from our own recent work. This demonstrates how dynamical systems approaches can contribute to understanding AI. Of course, the choice of results provides a biased view, and there are definitely many interesting results that one could discuss. Yet, as described in the introduction to the book, the book also serves as a final report of the first phase of a DFG priority program (SPP2298 Foundations of Deep Learning), to which all four authors of this chapter have contributed. Hence, this dual purpose is accounted for by our selection of recent works that emerged during our participation in the priority program. 

\subsection{The Basic Framework}
\label{cekk_sec:model}

A general neural network is a function
\begin{equation}\label{cekk_eq:neural_network}
	\Phi: \R^D \times \R^d \rightarrow \R^q, \quad (\theta, x) \mapsto \Phi(\theta,x),
\end{equation}
which depends on the input data $x$ and the parameters $\theta$. A typical neural network model is defined on a graph/network and consists of two main processes: information propagation and learning. The information propagation process describes the dynamics \emph{on} the network, that is, the flow of data through the network, while the learning process describes the dynamics \emph{of} the network, that is, the evolution of its weights and biases. As the dynamics of the network influences the dynamics on the network and vice versa, a neural network is an adaptive dynamical network.

\subsubsection*{Dynamics on the Network: Information Propagation}

For fixed weights, a neural network maps the state of the input $x\in\R^d$ to the state of the output $\Phi(\theta,x)$ and by that defines a map
\begin{equation}\label{cekk_eq:input_output_map}
	\Phi_\theta: \R^d \rightarrow \R^q, \quad x \mapsto \Phi_\theta(x) \coloneqq \Phi(\theta,x).
\end{equation}
Since information propagation describes a process on fixed parameters, we refer to it as the dynamics on the network.

\subsubsection*{Dynamics of the Network: Learning Process}

To approximate given data using the input-output map $\Phi_\theta$ of a neural network, the parameters~$\theta$ can be adjusted. Typically, for supervised learning, the training data consists of input-output pairs $(x, y)\in\R^d\times\R^q$, which are used to minimize a predefined loss function using (stochastic) gradient-based methods. Gradient descent defines an update map
$$\varphi: \mathbb R^D \to \mathbb R^D,\quad \theta \mapsto \varphi(\theta) \coloneqq \theta - \eta \nabla \mathcal L(\theta),$$
where $\eta>0$ is the learning rate and $\mathcal L$ the loss function. For stochastic gradient descent, this update map is replaced by a random update map
$$\varphi: \Omega \times \mathbb R^D \to \mathbb R^D,\quad (\omega,\theta) \mapsto \varphi_\omega(\theta),$$
where $(\Omega,\mathbb P)$ is a probability space. As the learning process iteratively modifies the underlying graph, we refer to it as the dynamics of the network, i.e., the learning process modifies the weights of the edges (or links) of the underlying network.

\subsection{Structure of the Chapter}

Section 2 investigates the process of information propagation through the network, i.e., the input-output map $\Phi_\theta$ as in equation~\eqref{cekk_eq:input_output_map}, as a dynamical system.
After introducing feed-forward neural networks (Section~\ref{cekk_sec:feedforward}), neural ordinary differential equations (ODEs) and delay differential equations (DDEs) are considered as corresponding infinite depth limits of deep neural networks (Section~\ref{cekk_sec:neuralODEs}). In the following, the main results from the works \cite{KuehnKuntz2023, KuehnKuntz2024, KuehnKuntz2025} are given and discussed: the \emph{universal embedding} for augmented neural ODEs representing arbitrary functions of given regularity (Section~\ref{cekk_sec:universality}), the classification of multilayer perceptrons and neural ODEs in terms of $C^k$-function classes (Section~\ref{cekk_sec:morse}) and the \emph{memory-dependence} in neural DDEs (Section~\ref{cekk_sec:delay}).

Section~\ref{cekk_sec:opti} studies the training process of a neural network as a dynamical system, characterizing the dynamical stability of minima and discussing the influence of stability on the problem of \emph{generalization} from training to test data. After formalizing the optimization task and distinguishing between the overdetermined and the overparameterized situation (Section~\ref{cekk_subs:opti}), we introduce a dynamical systems perspective on gradient descent and study stability for overdetermined problems (Section~\ref{cekk_subs:GDStab}). 
We then extend this analysis to the overparameterized setting and describe the edge of stability phenomenon, also in the context of possible explanations for implicit bias (Section~\ref{cekk_subs:GDOpar}). After introducing a random dynamical system framework for stochastic gradient descent, and briefly discussing the overdetermined setting (Section \ref{cekk_subs:RDS}), we present the key result from \cite{ChemnitzEngel24} characterizing the stability in stochastic gradient descent for the overparameterized setting via Lyapunov exponents of interpolation solutions (Section~\ref{cekk_sec:SGDOver}).

Section~\ref{cekk_sec:multi} discusses additional important topics from machine learning that are well-described by a dynamical perspective. First, we briefly discuss Boltzmann machines and Hopfield networks (Section~\ref{cekk_sec:Boltzmann}), that can be understood within the general framework of interacting particle systems (IPS). For IPS, we then consider mean-field limits that help to deploy dynamical rigorous tools for analysis (Section~\ref{cekk_sec:mean_field}). We state a result proving that (stochastic) IPS of Kuramoto-type on quite general graphs have mean-field limits~\cite{KuehnPulido} using the theory of digraph measures and graphops~\cite{KuehnGraphops,KuehnXu}. Many deep neural networks fall into this category, including recurrent neural networks and transformers (Section~\ref{cekk_sec:RNNs}). Finally, we also want to illustrate that many other questions in ML can highly benefit from a dynamical view. We cover generative models and their various forms, allowing for the application and further development of Markov chain theory, evolutionary game theory, or qualitative properties of stochastic differential equations (Section~\ref{cekk_sec:gen_models}). As a last step, we highlight that even the notion of backpropagation is just using a variational equation along trajectories in backward time, which is a very classical dynamics concept (Section~\ref{cekk_sec:backpropagation}).

\section{Neural Networks as Dynamical Systems}\label{cekk_sec:neural}

In this section, we focus on the information propagation process through the network, given by the input-output map \index{Input-Output Map of a Neural Network} $\Phi_\theta$ for fixed weights $\theta\in\R^D$, as defined in Section~\ref{cekk_sec:model}. The dynamics considered on the network can be defined on either a discrete or a continuous time scale. In the two upcoming sections, we introduce different neural network architectures. In Section \ref{cekk_sec:feedforward}, we focus on feed-forward neural networks (FNNs) \index{Feed-Forward Neural Network (FNN)} on a discrete time scale, for example, multilayer perceptrons (MLPs), residual neural networks (ResNets), and densely connected ResNets (DenseResNets). In Section \ref{cekk_sec:neuralODEs}, we introduce neural ordinary differential equations (neural ODEs) and neural delay differential equations (neural DDEs), which can be interpreted as an infinite depth limit of ResNets and DenseResNets.

Studying the input-output map $\Phi_\theta$ of a neural network as a dynamical system correlates with the question of which functions can be represented and approximated by the map $\Phi_\theta$. In Section \ref{cekk_sec:universality}, the concepts of universal embedding and universal approximation are introduced. Afterwards, the main results of the recent works \cite{KuehnKuntz2023}, \cite{KuehnKuntz2024}, and \cite{KuehnKuntz2025} are stated. In Section \ref{cekk_sec:morse}, the geometric structure of $\Phi_\theta$ is analyzed via Morse functions in the case of MLPs and neural ODEs. These properties are used to study in Section \ref{cekk_sec:delay} the interplay between the memory capacity of neural DDEs and their approximation capabilities. Overall, we highlight the interplay between the structural design and the dynamical properties of neural networks, which contributes to a more profound and theoretical understanding of the dynamics on the network. 

\subsection{Feed-Forward Neural Networks}
\label{cekk_sec:feedforward}

In the most studied neural network architectures, the neurons are structured in layers $h_l\in\R^{d_l}$ for $l\in\{0,\ldots,L\}$, building a feed-forward neural network. The number of layers, $L$, is called the depth of the network, and the maximal number of neurons per layer, $\max_{l\in\{0,\ldots,L\}}\{d_l\}$, is called the width of the network. The resulting neural network is a map $\Phi_\theta:\mathcal{X}\rightarrow \R^q$ with input layer $h_0 = x\in\mathcal{X}$, $\mathcal{X} \subset \R^d = \R^{d_0}$, output layer $h_L \in\R^{q} = \R^{d_L}$, and hidden layers $h_1,\ldots,h_{L-1}$. 

The most famous feed-forward neural network model is the multilayer perceptron \index{Multilayer Perceptron (MLP)} (MLP), which was already studied by Rosenblatt in 1957~\cite{Rosenblatt1957}. The update rule of  a general MLP is given by 
\begin{equation} \label{cekk_eq:mlp_updaterule}
	h_{l+1} = f_l(h_{l},\theta_l) \coloneqq \widetilde{W}_l\sigma_l(W_lh_{l}+b_l)+\tilde{b}_l 
\end{equation}
for $l \in \{0,\ldots,L-1\}$ with activation function $\sigma_l: \R^{m_l} \rightarrow \R^{m_l}$, where $W_l$, $\widetilde{W}_l$ are weight matrices and $b_l$, $\tilde{b}_l$ biases of appropriate dimensions. We abbreviate the parameters of layer $h_l$ by $\theta_l \coloneqq (W_l,\widetilde{W}_l,b_l,\tilde{b}_l)$ and all parameters by $\theta = (\theta_0,\ldots,\theta_{L-1})$. The activation function $\sigma_l$ is applied component-wise and often chosen as $\tanh$, soft-plus, sigmoid, or (normal, leaky, or parametric) $\textup{ReLU}$. The update rule~\eqref{cekk_eq:mlp_updaterule} includes both the case of an outer nonlinearity if $\widetilde{W}_l = \text{Id}$, $\tilde{b}_l = 0$ and the case of an inner nonlinearity if $W_l = \text{Id}$ and $b_l = 0$. For the analysis of MLPs in Section \ref{cekk_sec:morse}, we need the additional assumption that every component $[\sigma_l]_i$, $i \in \{1,\ldots,m_l\}$ of the activation function  is strictly monotone. 

\begin{definition}[Multilayer Perceptron (MLP)] \label{cekk_def:feedforward}
	The set of all MLPs $\Phi_\theta: \mathcal{X}\rightarrow \mathbb{R}^q$, $\mathcal{X} \subset \R^d$ open, with update rule~\eqref{cekk_eq:mlp_updaterule} and component-wise applied strictly monotone activation functions $[\sigma_l]_i\in C^k(\mathbb{R},\mathbb{R})$, $k \geq 0$, is denoted by $\Xi^k(\mathcal{X},\mathbb{R}^q) \subset C^k(\mathcal{X},\mathbb{R}^q)$.
\end{definition}

\begin{figure}[t]
	\centering
	\begin{subfigure}{\textwidth}
		\centering
		\begin{overpic}[width=0.7\linewidth,,tics=10]
			{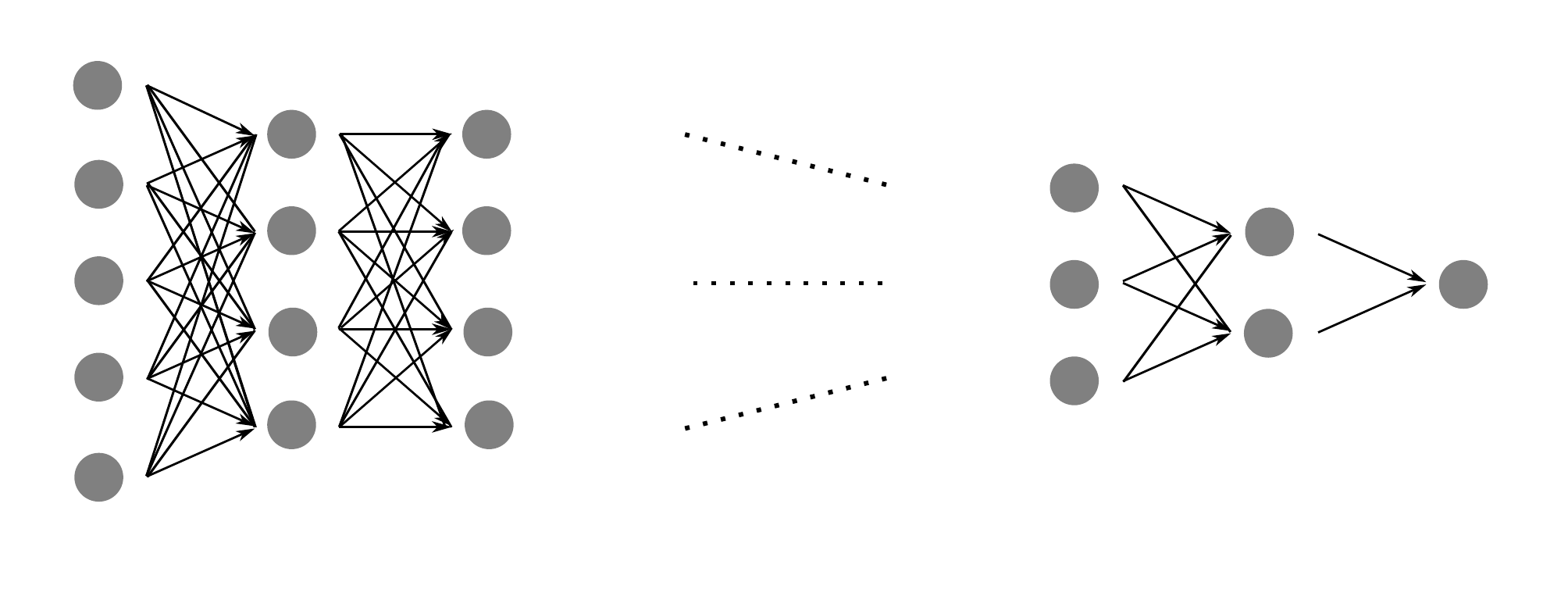}
			\put(5,1){$h_0$}
			\put(92,1){$h_L$}
		\end{overpic}
		\caption{Example of a non-augmented feed-forward neural network $\Phi \in \Xi^k_\text{N}(\R^5,\mathbb{R})$.}
	\end{subfigure}
	\begin{subfigure}{\textwidth}
		\centering
		\vspace{1mm}
		\begin{overpic}[width=0.7\linewidth,,tics=10]
			{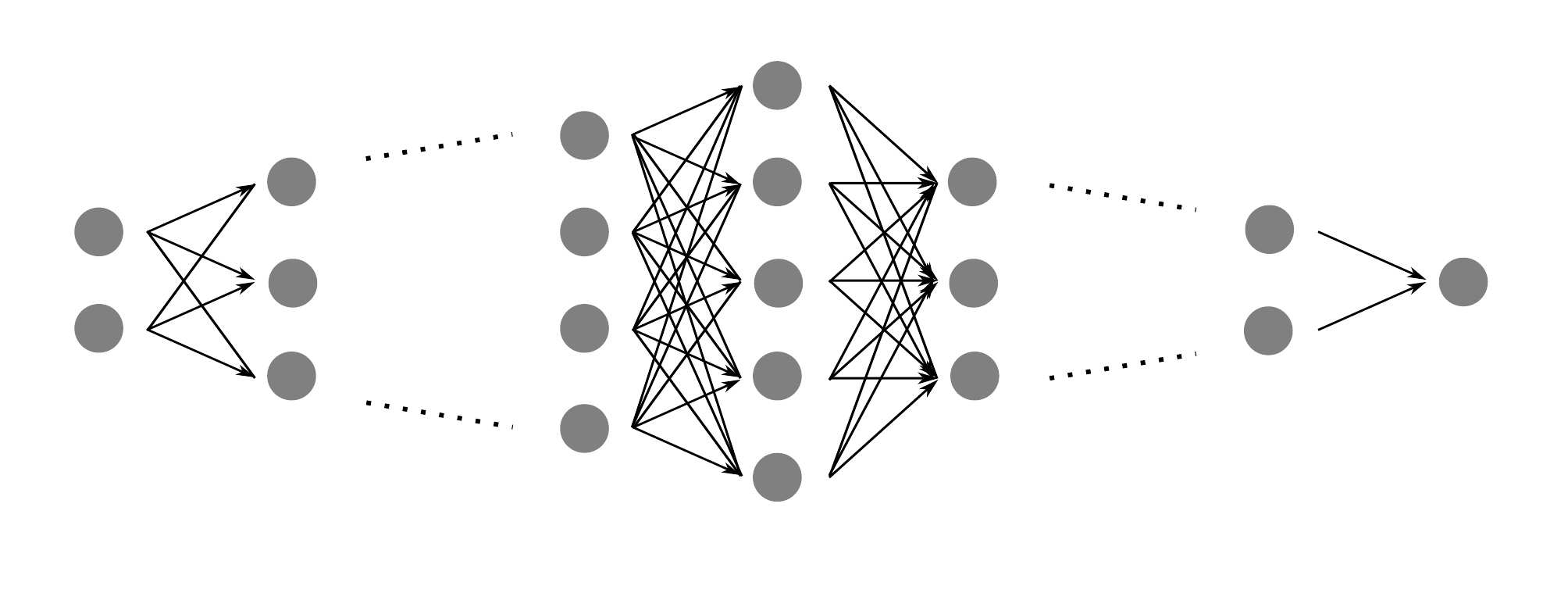}
			\put(5,1){$h_0$}
			\put(48,1){$h_{l_\text{max}}$}
			\put(92,1){$h_L$}
		\end{overpic}
		\caption{Example of an augmented feed-forward neural network $\Phi \in \Xi^k_\text{A}(\R^2,\mathbb{R})$.}
	\end{subfigure}
	\begin{subfigure}{\textwidth}
		\centering
		\vspace{1mm}
		\begin{overpic}[width=0.7\linewidth,,tics=10]
			{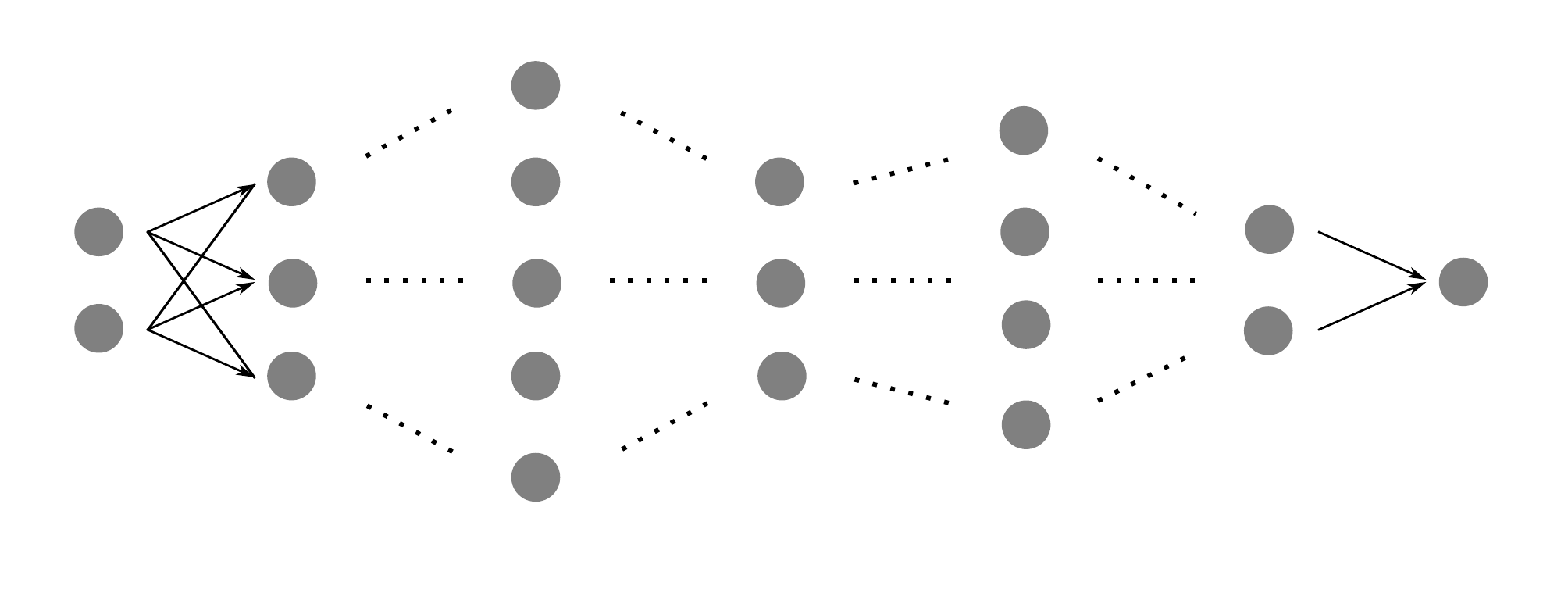}
			\put(5,1){$h_0$}
			\put(33,1){$h_i$}
			\put(48,1){$h_{l_\textup{B}}$}
			\put(64,1){$h_j$}
			\put(92,1){$h_L$}
		\end{overpic}
		\caption{Example of a feed-forward neural network with a bottleneck $\Phi \in \Xi^k_\text{B}(\R^2,\mathbb{R})$.}
	\end{subfigure}
	\caption{Structure of non-augmented, augmented, and bottleneck scalar feed-forward neural networks with $q = 1$. In general, the output of a feed-forward neural network can be multi-dimensional. The Figures~1(a)-(c) are adapted from \cite[Figure~3.2(a)-(c)]{KuehnKuntz2024}.}
	\label{cekk_fig:FNNclassification}
\end{figure}

Depending on the dimensions $d_l$ of the layers $h_l$, and the dimensions $m_l$ of the activation functions $\sigma_l$, every feed-forward neural network can be categorized into three different types of architectures. For presentation purposes, we assume in the following that $d_{l} = m_l$ for all $l \in \{0,\ldots,L-1\}$, the general case is discussed in~\cite{KuehnKuntz2024}.
\begin{itemize}
	\item \textbf{Non-augmented FNN:} \index{Non-Augmented Neural Network Architecture} the dimensions of the layers are monotonically decreasing from $h_0$ to $h_L$, i.e., $d_{l-1}\geq d_{l}$ for all $l\in\{1,\ldots,L\}$.
	\item \textbf{Augmented FNN:} \index{Augmented Neural Network Architecture} the layer of maximal dimension $h_{l_\text{max}}$ is wider than the input layer $h_0$, i.e., $d_{l_\text{max}}>d_0$, and it holds $d_{l-1}\leq d_{l}$ for all $l\in\{1,\ldots,l_\text{max}\}$ and $d_{l-1}\geq d_{l}$ for all $l\in\{l_\text{max}+1,\ldots,L\}$.
	\item \textbf{FNN with a bottleneck:} there exists a layer $h_{l_\text{B}}$, called a bottleneck, and two layers $h_i$ and $h_j$ such that $0\leq i<l_\text{B}<j\leq L$ and  $d_i>d_{l_\text{B}}$, $d_{l_\text{B}}<d_j$.
\end{itemize}
The three different types of architectures build a complete disjoint subdivision of all feed-forward neural networks structured in layers $h_l\in\R^{d_l}$. In the case of MLPs, we denote the non-augmented, augmented, and bottleneck architectures by $\Xi^k_\text{N}(\mathcal{X},\mathbb{R}^q)$, $\Xi^k_\text{A}(\mathcal{X},\mathbb{R}^q)$ and $\Xi^k_\text{B}(\mathcal{X},\mathbb{R}^q)$ respectively. In Figure \ref{cekk_fig:FNNclassification}, examples of a non-augmented, an augmented, and an FNN with a bottleneck are shown for $q = 1$. In our analysis of MLPs in Section \ref{cekk_sec:morse}, we mainly focus on single output components, corresponding to the scalar case $q=1$.

A more advanced class of feed-forward neural networks are residual neural networks \index{Residual Neural Network (ResNet)} (ResNets)~\cite{He2016}. In ResNet architectures, all layers have the same width $d = d_l$, $l\in\{0,\ldots,L\}$ and the update rule is of the form
\begin{equation}\label{cekk_eq:ResNet}
	h_{l+1} = h_l + f_l(h_{l},\theta_l)
\end{equation} 
for $l \in \{0,\ldots,L-1\}$. In comparison to the update rule~\eqref{cekk_eq:mlp_updaterule} of MLPs, the ResNet update rule~\eqref{cekk_eq:ResNet} allows for an additional shortcut/skip connection via the summand $h_l$. The additional skip connection is advantageous for gradient-based training methods of deep neural networks with a large number of layers $L \gg 1$ \cite{He2016}. Furthermore, if the update rule is independent of the layer, i.e., $f_\text{ResNet}(\cdot,\cdot) = f_l(\cdot,\cdot)$ for all $l \in \{0,\ldots,L-1\}$, the ResNet update rule~\eqref{cekk_eq:ResNet} motivates the study of neural ODEs, which are introduced in Section \ref{cekk_sec:neuralODEs}.

Even more evolved feed-forward architectures are for example U-Nets \index{U-Net} (cf.~\cite{Ronneberger2015}) and DenseNets \index{DenseNet} (cf.~\cite{Huang2017}), which allow for additional inter-layer connections. In that way, information or features of previous layers can be re-introduced in future layers. Furthermore, inter-layer connections can mitigate problems during the training process, like vanishing gradients \cite{Huang2017}. In this work, we focus on neural DDEs, which can be interpreted as the continuous time analogue of densely connected feed-forward neural networks with a ResNet structure (DenseResNets). \index{Residual Neural Network (ResNet)!Densely Connected ResNet (DenseResNet)} In DenseResNet architectures, all layers have the same  width~$d = d_l$, $l\in\{0,\ldots,L\}$ and the layer $h_{l+1}$ can depend on all previous layers, resulting in the update rule
\begin{equation}\label{cekk_eq:DenseNet}
	h_{l+1} = h_{l} + f_{\textup{Dense},l}(h_{l}, h_{l-1}, \ldots, h_0, \theta_{\textup{Dense},l})
\end{equation}
for $l \in \{0,\ldots,L-1\}$. Hereby $\theta_{\textup{Dense},l}$ stores all weights and biases between the layers $h_{l+1}$ and $h_{l+1-i}$ for every $i \in \{1,\ldots,l+1\}$. It is important to note that a DenseResNet is structurally different from a FNN or ResNet with fully connected layers, where all weight matrices have only non-zero entries. In Figure \ref{cekk_fig:DenseResNet}, all possible inter-layer connections of the DenseResNet with update rule \eqref{cekk_eq:DenseNet} are visualized.

\begin{figure}[h]
	\centering
	\begin{overpic}[scale = 0.45,,tics=10]
		{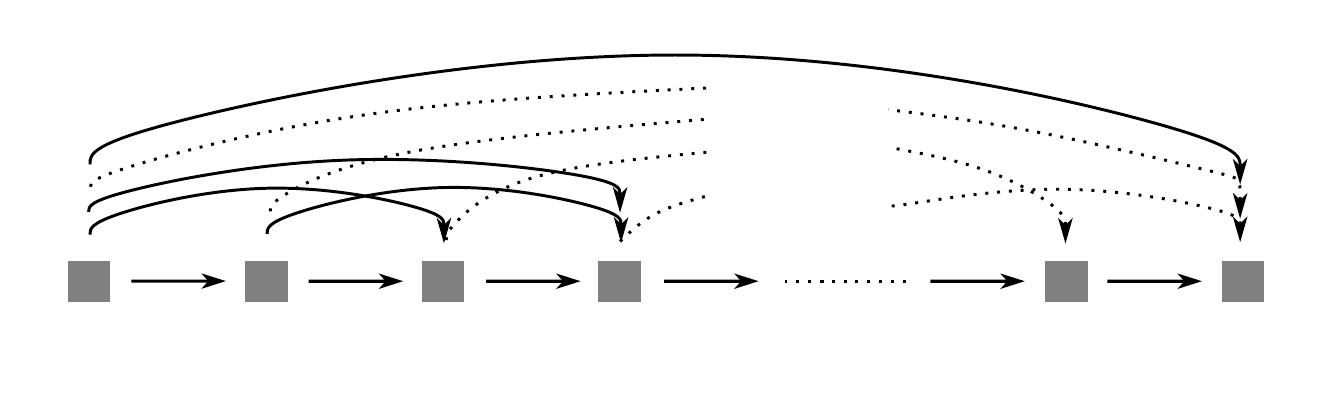}
		\put(5,1){$h_0$}
		\put(19,1){$h_1$}
		\put(32,1){$h_2$}
		\put(45,1){$h_3$}
		\put(77,1){$h_{L-1}$}
		\put(92,1){$h_L$}
	\end{overpic}
	\caption{Structure of a densely connected residual neural network (DenseResNet). In contrast to Figure~\ref{cekk_fig:FNNclassification}, every vertex represented by a square corresponds to a layer $h_l \in \R^{d_l}$. The figure is adapted from \cite[Figure 1.1(b)]{KuehnKuntz2025}.}
	\label{cekk_fig:DenseResNet}
\end{figure}
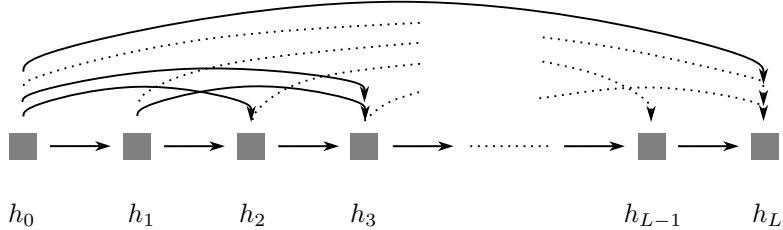

\subsection{Neural ODEs and Neural DDEs}
\label{cekk_sec:neuralODEs}

Besides neural networks on a discrete time scale, such as the feed-forward neural networks introduced in Section \ref{cekk_sec:feedforward}, architectures on a continuous time scale exist. First, we introduce neural ordinary differential equations (neural ODEs), \index{Neural Ordinary Differential Equation (Neural ODE)} which can be interpreted as an infinite depth limit of residual neural networks \index{Infinite Depth Limit} \cite{Chen2018, Weinan2017, He2016, KuehnKuntz2023}. On the one side, continuous time models like neural ODEs have the advantage of  constant memory cost during the training process, on the other side they provide insight into the dynamics of deep neural networks, as motivated in the following~\cite{Chen2018}. The iterative update rule \eqref{cekk_eq:ResNet} of ResNets with $f_\text{ResNet}(\cdot,\cdot) = f_l(\cdot,\cdot)$ for all $l \in \{0,\ldots,L-1\}$ can be obtained as an Euler discretization of the initial value problem~(IVP)
\begin{equation}\label{cekk_eq:NeuralODE1}
	\frac{\dd h}{\dd t} = f_\text{ODE} (h(t),\theta(t)), \quad h(0) = x,
\end{equation}
on the time interval $[0,T]$ with step size $\delta = \frac{T}{L}$ and a vector field $f_\text{ODE}(\cdot,\cdot)\coloneqq \frac{1}{\delta}f_\text{ResNet}(\cdot,\cdot)$. Hereby, the solution $h$ of the IVP can be interpreted as the hidden layers, and the function $\theta$ as the weights. The initial condition of the ODE is the first layer $h_0 = x\in\R^d$, and the output layer $h_L$ corresponds to the time-$T$ map $h(T)\in\R^d$ of the IVP (cf.\ \cite{Guckenheimer2002}). Typical learning algorithms are applied to stationary parameters of neural networks rather than to parameter functions. To be able to optimize with respect to stationary parameters, the IVP \eqref{cekk_eq:NeuralODE1} can be rewritten in the form
\begin{equation}\label{cekk_eq:NeuralODE2}
	\frac{\dd h}{\dd t} = f(t,h(t)), \quad h(0) = x,
\end{equation}
with a parameter-dependent non-autonomous vector field $f$. As the discretization error of an Euler discretization depends on the step size $\delta = \frac{T}{L}$, the deeper the ResNet, the better the neural ODE dynamics can approximate the ResNet dynamics.

In analogy to neural ODEs, we define neural delay differential equations (neural DDEs), \index{Neural Delay Differential Equation (Neural DDE)} which can be interpreted as an infinite depth limit of DenseResNets. Under some assumptions on the vector field (as specified in \cite{KuehnKuntz2025}), the update rule \eqref{cekk_eq:DenseNet} can be obtained as an Euler discretization of the delay differential equation (DDE)
\begin{equation}\label{cekk_eq:NeuralDDE}
	\frac{\dd h(t)}{\dd t} = F(t,h_t), \qquad h(s) = c_x(s) \quad \text{for } s \in [-\tau,0],
\end{equation}
with delay $\tau \geq 0$, delayed function $h_t \in C([-\tau,0],\R^d)$ defined by $h_t(s) = h(t+s)$ for $s\in[-\tau,0]$, parameter-dependent non-autonomous vector field $F$ and constant initial data $c_x:[-\tau,0]\rightarrow \mathbb{R}^d$ with $c_x(t) = x$. The input function $c_x$ of the neural DDE is specified by the input data $h_0 = x$ of the corresponding DenseResNet, and the output is given by the time-$T$ map $h(T)$ \cite{KuehnKuntz2025,Zhu2021}. As the state space $C([-\tau,0],\R^d)$ of the dynamical system generated by a DDE \eqref{cekk_eq:NeuralDDE} is infinite-dimensional, it is necessary to specify an initial function over the time interval $[-\tau,0]$ instead of one initial value at $t = 0$ for ODEs.

Neural ODEs and neural DDEs defined by  \eqref{cekk_eq:NeuralODE2} and \eqref{cekk_eq:NeuralDDE} are limited to modeling data with the same input and output dimensions $d$. To overcome this restriction, two affine linear layers 
\begin{equation}
	\lambda: \R^d \rightarrow \R^m, \; \lambda(x) = Wx+b, \qquad \tilde\lambda: \R^m \rightarrow \R^q, \; \tilde \lambda(x) = \widetilde Wx+ \tilde b,
\end{equation}
can be added before and after the solution of an $m$-dimensional neural ODE or DDE. A neural network $\Phi_\theta: \mathcal{X}\rightarrow\R^q$ is now defined as follows: given some initial data $x\in\mathcal{X} \subset \R^d$, the transformed state $\lambda(x)\in\R^m$ is used as point initial condition for the ODE~\eqref{cekk_eq:NeuralODE2} and $c_{\lambda(x)}\in C([-\tau,0],\R^m)$ is used as constant initial function for the DDE~\eqref{cekk_eq:NeuralDDE}. Then, the ODE/DDE is solved over the time interval $[0,T]$. The output of the neural ODE/DDE architecture is defined as
\begin{equation}\label{cekk_eq:neuralODEDDE}
	\Phi_\theta(x) = \tilde \lambda(h(T)) \in \R^q,
\end{equation}
which is the second linear transformation $\tilde \lambda$ applied to the time-$T$ map $h(T)$. The structure of a neural ODE/DDE with two additional affine linear layers is visualized in Figure \ref{cekk_fig:neuralODEDDE}. The parameters~$\theta$ of the continuous time neural network~\eqref{cekk_eq:neuralODEDDE} are given by the weights and biases $W,\widetilde W, b, \tilde b$ and the parameters of the vector field $f$ or~$F$. 

In the following, we do not restrict our analysis to a fixed parametrization of the vector field, but consider general, non-parameterized neural ODE/DDE architectures. To that purpose, we consider for neural ODEs the space $C^{0,k}([0,T]\times \mathbb{R}^m,\mathbb{R}^m)$, $k \geq 1$, which consists of all functions that are continuous in the first variable and $k$ times continuously differentiable in the second variable. For neural DDEs, we consider the space $C^{0,k}_b([0,T]\times C([-\tau,0],\R^m),\mathbb{R}^m)$, $k \geq 1$, of functions that are continuous in the first variable and have bounded continuous derivatives up to order $k$ with respect to the second variable.

\begin{definition}[Neural ODE] \label{cekk_def:NODE}
	For $k \geq 1$, the set of all neural ODE architectures $\Phi_\theta: \mathcal{X}\rightarrow \mathbb{R}^q$, with $\mathcal{X} \subset \R^d$ open, defined by~\eqref{cekk_eq:neuralODEDDE} and based on the initial value problem \eqref{cekk_eq:NeuralODE2} with non-parameterized vector field $f\in C^{0,k}([0,T]\times \mathbb{R}^m,\mathbb{R}^m)$, is denoted by $\textup{NODE}^k(\mathcal{X},\R^q) \subset C^k(\mathcal{X},\R^q)$. 
\end{definition}

\begin{definition}[Neural DDE] \label{cekk_def:NDDE}
	For $k \geq 1$, the set of all neural DDE architectures $\Phi_\theta: \mathcal{X}\rightarrow \mathbb{R}^q$, with $\mathcal{X}\subset \R^d$ open, defined by~\eqref{cekk_eq:neuralODEDDE} and based on the DDE \eqref{cekk_eq:NeuralDDE} with non-parameterized vector field $F\in C^{0,k}_b([0,T]\times C([-\tau,0],\R^m),\mathbb{R}^m)$, is denoted by $\textup{NDDE}_\tau^k(\mathcal{X},\R^q) \subset C^k(\mathcal{X},\R^q)$. 
\end{definition}

By definition of the vector fields in~\eqref{cekk_eq:NeuralODE2} and \eqref{cekk_eq:NeuralDDE}, it follows that every neural DDE with delay $\tau = 0$ is a neural ODE. Neural ODEs and neural DDEs can also be well-defined under weaker regularity assumptions on the vector field than in Definitions~\ref{cekk_def:NODE} and \ref{cekk_def:NDDE}, as shown in~\cite{KuehnKuntz2023, KuehnKuntz2025}.
Depending on the dimensions $d$ and $m$, we can distinguish different types of architectures of neural ODEs and neural DDEs:
\begin{itemize}
	\item \textbf{Non-augmented neural ODE / DDE:} \index{Non-Augmented Neural Network Architecture} it holds $m\leq d$ and the weight matrices $W,\widetilde W$ have full rank.
	\item \textbf{Augmented neural ODE / DDE:} \index{Augmented Neural Network Architecture}  it holds $m> d$ and the weight matrices $W,\widetilde W$ have full rank.
	\item \textbf{Degenerate neural ODE / DDE:} one of the weight matrices $W,\widetilde W$ is singular.
\end{itemize}
We denote non-augmented architectures by $\textup{NODE}^k_\text{N}(\mathcal{X},\mathbb{R}^q)$ / $\textup{NDDE}^k_{\tau,\text{N}}(\mathcal{X},\mathbb{R}^q)$, augmented architectures by $\textup{NODE}^k_{\text{A}}(\mathcal{X},\mathbb{R}^q)$ / $\textup{NDDE}^k_{\tau,\text{A}}(\mathcal{X},\mathbb{R}^q)$ and degenerate architectures by $\textup{NODE}^k_{\text{D}}(\mathcal{X},\mathbb{R}^q)$ / $\textup{NDDE}^k_{\tau,\text{D}}(\mathcal{X},\mathbb{R}^q)$, respectively. In Figure \ref{cekk_fig:neuralODEDDE}, the structure of a non-augmented neural ODE and an augmented neural DDE is visualized. By using the theory of continuous-time dynamical systems, we can gain a better understanding of the theoretical properties of neural ODEs and neural DDEs. Via approximation results, these insights can be transferred to deep ResNets defined in~\eqref{cekk_eq:ResNet} and deep DenseResNets defined in \eqref{cekk_eq:DenseNet}. In summary, we can immediately use the dynamical systems approach of ODEs and DDEs to understand classes of neural network architectures in the infinite depth limit.

\begin{figure}[t]
	\centering
	\begin{subfigure}{\textwidth}
		\centering
		\begin{overpic}[scale = 0.38,,tics=10]
			{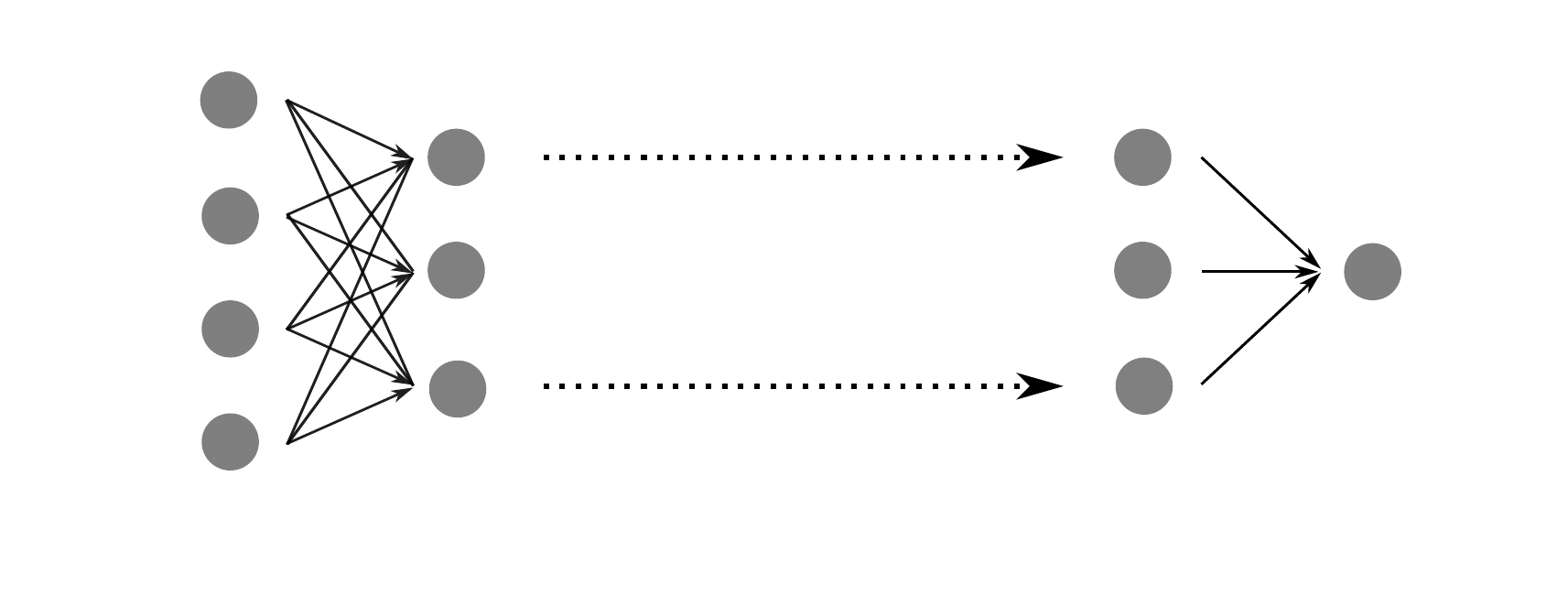}
			\put(12,1){$x$}
			\put(25,1){$\lambda(x)$}
			\put(69,1){$h(T)$}
			\put(84,1){$\Phi_\theta(x)$}
			\put(46,19){\normalsize{ODE}}
		\end{overpic}
		\caption{Example of a non-augmented neural ODE $\Phi \in \text{NODE}^k_{\text{N}}(\R^4,\mathbb{R})$ with $m \leq d$.}
	\end{subfigure}
	\begin{subfigure}{\textwidth}
		\centering
		\vspace{2mm}
		\begin{overpic}[scale = 0.38,,tics=10]
			{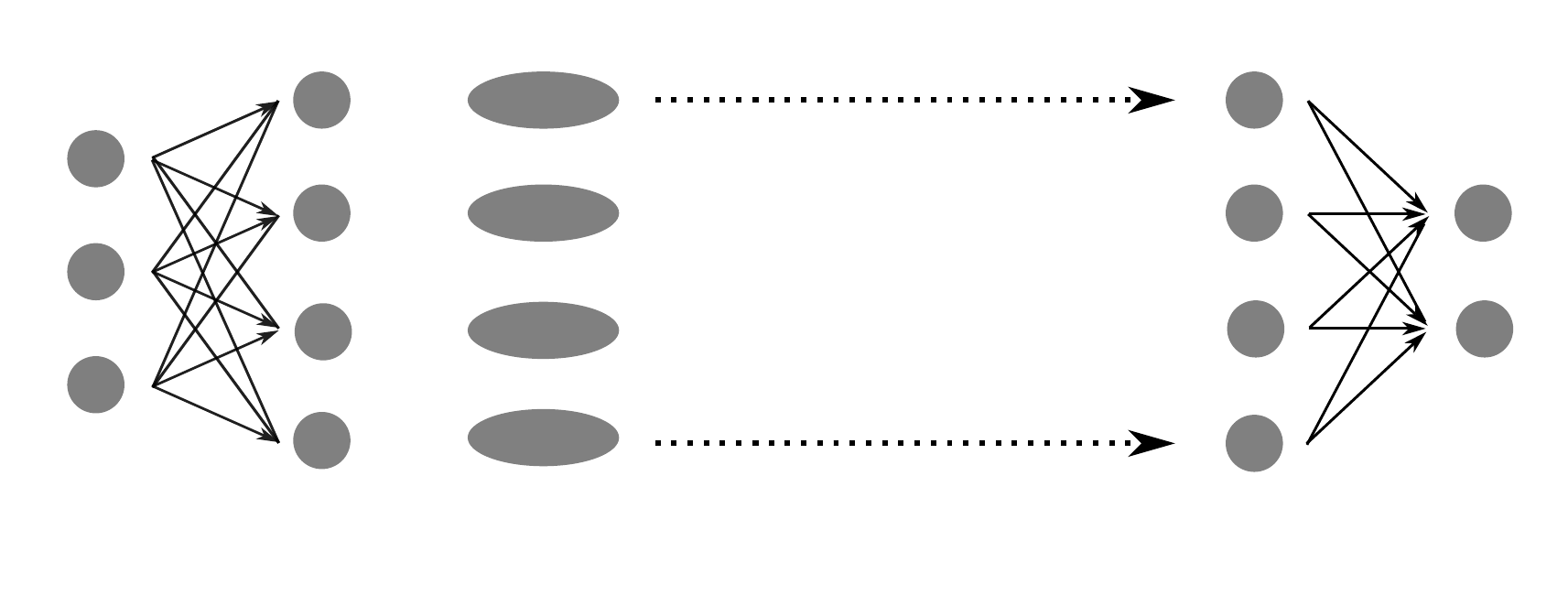}
			\put(5,1){$x$}
			\put(17,1){$\lambda(x)$}
			\put(31,1){$c_{\lambda(x)}$}
			\put(76,1){$h(T)$}
			\put(91,1){$\Phi_\theta(x)$}
			\put(53,19){\normalsize{DDE}}
		\end{overpic}
		\caption{Example of an augmented neural DDE $\Phi \in \text{NDDE}^k_{\tau,\text{A}}(\R^3,\mathbb{R}^2)$ with $m > d$.}
	\end{subfigure}
	\caption{Structure of the general neural ODE/DDE architecture with two affine linear layers. Figure~3(a) is adapted from \cite[Figure 4.1(a)]{KuehnKuntz2024} and Figure 3(b) is adapted from \cite[Figure 2.4(b)]{KuehnKuntz2025}.}
	\label{cekk_fig:neuralODEDDE}
\end{figure}

\subsection{Universal Embedding and Universal Approximation}
\label{cekk_sec:universality}

In supervised learning, neural networks are used to learn a given ground truth data
\begin{equation}
	\Psi: \mathcal{X} \rightarrow \mathcal{Y}, \quad x \mapsto y\coloneqq \Psi(x),
\end{equation}
with topological spaces $\mathcal{X}$ and $\mathcal{Y}$. Hence, a desired property of neural networks $\Phi_\theta:\mathcal{X}\rightarrow\mathcal{Y}$ is universal approximation, i.e., the property to approximate any function $\Psi$ of a given function space arbitrarily well. In the context of the function class $C^k(\mathcal{X},\mathcal{Y})$ of $k$ times continuously differentiable functions mapping from $\mathcal{X}$ to~$\mathcal{Y}$, universal approximation is defined as follows.

\begin{definition}[Universal Approximation \index{Universal Approximation} \cite{KuehnKuntz2023}] \label{cekk_def:universalapproximation}
	A neural network $\Phi_\theta: \mathcal{X} \rightarrow \mathcal{Y}$ with parameters $\theta$, topological space $\mathcal{X}$ and metric space $\mathcal{Y}$ has the universal approximation property with respect to the space $C^k(\mathcal{X},\mathcal{Y})$, $k \geq 0$, if for every $\varepsilon >0$ and for each function $\Psi \in C^k(\mathcal{X},\mathcal{Y})$, there exists a choice of parameters~$\theta$, such that $\textup{dist}_\mathcal{Y}(\Phi_\theta(x),\Psi(x)) < \varepsilon$ for all $x \in \mathcal{X}$.
\end{definition}

For many feed-forward neural network architectures, the universal approximation property can be obtained by increasing the width, the depth, or the number of parameters of the network~\cite{Cybenko1989, Hornik1989, Lin2018, Pinkus1999, Schaefer2006}. In fact, it was proven as early as the 1950s and 1960s that superpositions and/or iteration of mappings can be used to approximate functions and sequences using very simple building blocks~\cite{Kolmogorov1,Sharkovskii1}. Also, for neural ODEs, universal approximation theorems can be proven, especially by augmenting the phase space~\cite{Dupont2019, Kidger2022, KuehnKuntz2023, RuizBalet2023, Zhang2020}. Besides universal approximation, it is mathematically often more transparent to study universal embedding, i.e., the exact representation of functions by a neural network. Analyzing the geometry of the exact output of a neural network, depending on the structure, can provide a fundamental understanding of why certain architectures outperform others.

\begin{definition}[Universal Embedding \index{Universal Embedding} \cite{KuehnKuntz2023}]\label{cekk_def:universalembedding}
	A neural network $\Phi_\theta: \mathcal{X} \rightarrow \mathcal{Y}$ with parameters $\theta$ and topological spaces $\mathcal{X}$ and $\mathcal{Y}$ has the universal embedding property with respect to the space $C^k(\mathcal{X},\mathcal{Y})$, $k \geq 0$, if for every function $ \Psi \in C^k(\mathcal{X},\mathcal{Y})$, there exists a choice of parameters~$\theta$, such that $\Phi_\theta(x) = \Psi(x)$ for all $x \in \mathcal{X}$.
\end{definition}

By definition, every neural network $\Phi_\theta$ that has the universal embedding property also has the universal approximation property. In the upcoming Sections \ref{cekk_sec:morse} and \ref{cekk_sec:delay}, we present the main results of the recent works \cite{KuehnKuntz2024} and \cite{KuehnKuntz2025} about the geometric structure of the input-output map of MLPs, neural ODEs, and neural DDEs and their implications on universal embedding and universal approximation. In \cite{KuehnKuntz2023}, the embedding capability of neural ODEs is studied via iterative functional equations, differential geometry, suspension flows, and dynamical systems theory. In contrast to MLPs, it is possible to prove a universal embedding result for neural ODEs with a general non-autonomous vector field $f$ if the dimension of the vector field is sufficiently large.

\begin{theorem}[Universal Embedding for Augmented Neural ODEs \cite{KuehnKuntz2023}]\label{cekk_th:universalembedding}    
	The neural ODE architecture $\textup{NODE}_\textup{A}^k(\mathcal{X},\R^q)$, $\mathcal{X} \subset \R^d$ open, $k\geq 1$, has for $m \geq d+q$ the universal embedding property with respect to the space $C^k(\mathcal{X},\R^q)$.
\end{theorem}

In the case of scalar neural ODEs, i.e., neural ODEs with $q=1$, Theorem \ref{cekk_th:universalembedding} holds for all augmented neural ODEs, as the condition $m\geq d+q =d+1$ is equivalent to the definition of an augmented neural ODE in Section \ref{cekk_sec:neuralODEs}, given by $m> d$.

\subsection{Analysis of the Geometric Structure via Morse Functions}
\label{cekk_sec:morse}

In this section, we restrict the analysis to scalar MLPs \index{Multilayer Perceptron (MLP)} and scalar neural ODEs, \index{Neural Ordinary Differential Equation (Neural ODE)} i.e., the output dimension is $q = 1$. In the case of neural networks with $q>1$, our results apply to single components of the output. We aim to analyze the input-output map $\Phi_\theta:\mathcal{X}\rightarrow \R$, $\mathcal{X} \subset \R^d$ of scalar MLPs and scalar neural ODEs with respect to the property of the existence and regularity of critical points, which is often characterized via Morse functions. 

\begin{definition}[Morse Function~\cite{Hirsch1976, Morse1934}] \label{cekk_def:morse}
	A map $\Psi \in C^2(\mathcal{X}, \mathbb{R})$, $\mathcal{X}\subset \R^d$ open, is called a Morse function \index{Morse Function} if all critical points of $\Psi$ are non-degenerate, i.e., for every critical point $p \in \mathcal{X}$ defined by a zero gradient $\nabla \Psi(p) = 0 \in \mathbb{R}^d$, the Hessian matrix $H_{\Psi}(p)\in \mathbb{R}^{d \times d}$ is non-singular.  
\end{definition}

In the Banach space of smooth scalar bounded functions, Morse functions are dense, i.e., it is a generic property of a smooth function to be Morse. To not only distinguish if a function is Morse or not, but also to classify the existence of critical points, we define three different subspaces of the function space $C^k(\mathcal{X},\mathbb{R})$, $\mathcal{X}\subset \R^d$.

\begin{definition}[$C^k$-Function Classes] \label{cekk_def:subsets}
	For $k \geq 2$, depending on the existence and regularity of critical points, we define three subspaces of $C^k(\mathcal{X},\mathbb{R})$, $\mathcal{X}\subset \R^d$ open:
	\begin{flalign*}
		\quad (\mathcal{C}1)^k(\mathcal{X},\mathbb{R}) \coloneqq &\left\{ \Psi \in C^k(\mathcal{X}, \mathbb{R}) : \nabla  \Psi(x) \neq 0\;  \forall x \in \mathcal{X}\right\}, && \\
		\quad (\mathcal{C}2)^k (\mathcal{X},\mathbb{R})\coloneqq &\left\{ \Psi \in C^k(\mathcal{X}, \mathbb{R}) : \left(\exists \, p:  \nabla  \Psi(p) = 0\right) \; \text{and} \right. && \\
		&\left. \; \left(\nabla  \Psi(q) = 0 \; \Rightarrow \; H_{ \Psi}(q) \textup{ is non-singular}\right) \right\}, && \\
		(\mathcal{C}3)^k(\mathcal{X},\mathbb{R}) \coloneqq &\left\{ \Psi \in C^k(\mathcal{X}, \mathbb{R}) : \exists \, p:  \left(\nabla  \Psi(p) = 0 \; \text{and} \; H_{ \Psi}(p) \textup{ is singular}\right) \right\}. &&
	\end{flalign*}
	The defined subspaces are non-empty and build a disjoint subdivision of  $C^k(\mathcal{X},\mathbb{R})$, i.e., $C^k(\mathcal{X},\mathbb{R}) = (\mathcal{C}1)^k(\mathcal{X},\mathbb{R})\, \dot{\cup} \, (\mathcal{C}2)^k(\mathcal{X},\mathbb{R}) \, \dot{\cup}\, (\mathcal{C}3)^k(\mathcal{X},\mathbb{R})$.
\end{definition}

By definition, all functions of the classes $(\mathcal{C}1)^k(\mathcal{X},\mathbb{R})$ and $(\mathcal{C}2)^k(\mathcal{X},\mathbb{R})$ are Morse functions, and no function of the class $(\mathcal{C}3)^k(\mathcal{X},\mathbb{R})$ is a Morse functions. One-dimensional examples of the defined function classes are the following: $(\mathcal{C}1)^k(\R,\mathbb{R})$ includes all strictly monotone functions, $(\mathcal{C}2)^k(\R,\mathbb{R})$ includes the quadratic function $x \mapsto x^2$ and $(\mathcal{C}3)^k(\R,\mathbb{R})$ contains all higher-order monomials $x^n$ with $n\geq3$. For \mbox{$d = 2$}, two examples of Morse functions of the class $(\mathcal{C}2)^k(\R^2,\mathbb{R})$, represented by the circle and the XOR data set, are visualized in Figure \ref{cekk_fig:morsefunction}.

\begin{figure}[h]
	\centering
	\begin{subfigure}{0.45\textwidth}
		\centering
		\begin{overpic}[scale = 0.6,,tics=10]
			{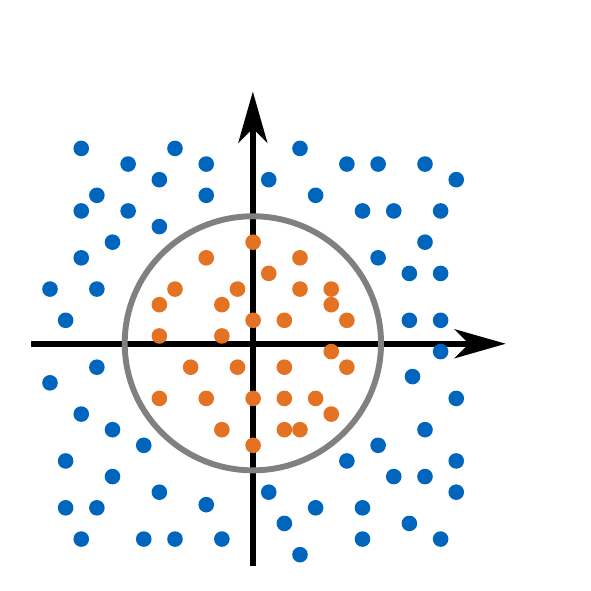}
			\put(88,40){$x_1$}
			\put(39,89){$x_2$}
		\end{overpic}
		\caption{Circle data set with underlying Morse function $\Psi(x) = x_1^2+x_2^2-1.$}
	\end{subfigure}
	\begin{subfigure}{0.05\textwidth}
		\textcolor{white}{.}
	\end{subfigure}
	\begin{subfigure}{0.45\textwidth}
		\centering
		\begin{overpic}[scale = 0.6,,tics=10]
			{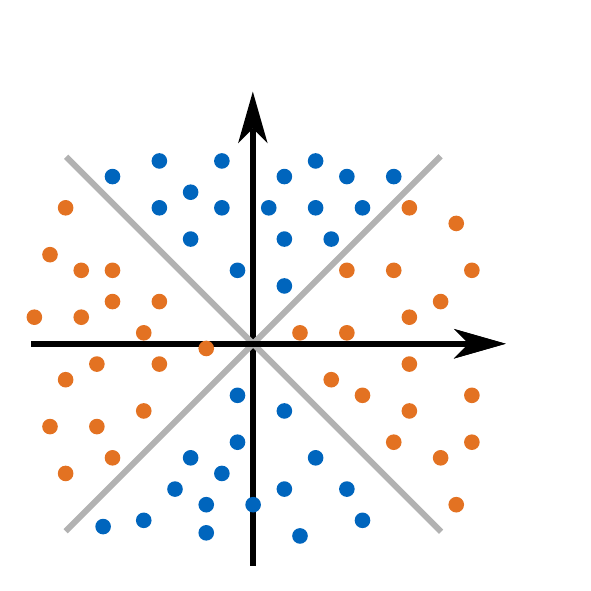}
			\put(88,40){$x_1$}
			\put(39,89){$x_2$}
		\end{overpic}
		\caption{XOR data set with underlying Morse function $\Psi(x) = x_2^2-x_1^2.$}
	\end{subfigure}
	\caption{Two-dimensional datasets with underlying Morse function $\Psi \in (\mathcal{C}2)^k(\R^2,\R)$: points with $\Psi(x)<0$ are visualized in orange, and points with $\Psi(x)>0$ in blue.}
	\label{cekk_fig:morsefunction}
\end{figure}

In the context of scalar MLPs and scalar neural ODEs, Morse functions play an essential role. This is due to the fact that, depending on the type of architecture (non-augmented, augmented, bottleneck/degenerate) defined in Sections \ref{cekk_sec:feedforward} and \ref{cekk_sec:neuralODEs}, different results regarding the existence and regularity of critical points are obtained. The upcoming theorem focuses on MLPs, where all weight matrices have full rank. If an MLP has a singular weight matrix, we can identify via an explicit construction an equivalent MLP, called the normal form, that has the same input-output map $\Phi_\theta$, but all weight matrices have full rank  \cite{KuehnKuntz2024}. By applying the algorithm to obtain the normal form with only full rank weight matrices, the underlying structure and hence also the classification (non-augmented, augmented, bottleneck) can change.

\begin{theorem}[Classification of scalar MLPs and scalar Neural ODEs \cite{KuehnKuntz2024}] \label{cekk_th:classification}
	Let $\mathcal{X}\subset \R^d$ open and consider scalar MLPs $\Phi_\theta \in \Xi^k(\mathcal{X},\R)$ with only full rank matrices and weight space $\mathbb{W} = (W_0,\widetilde{W}_0,\ldots,  W_{L-1}, \widetilde W_{L-1})$ and scalar neural ODEs $\Phi_\theta \in \textup{NODE}^k(\mathcal{X},\R)$ with weight space $\mathbb{V} = (W,\widetilde W)$. Then, depending on the underlying structure of the network (non-augmented, augmented, bottleneck/degenerate), the classification results in Table \ref{cekk_tab:results} hold.
\end{theorem}

From the results of Table \ref{cekk_tab:results}, we conclude that MLPs and neural ODEs show in the non-augmented and the augmented case the same geometric properties with respect to the existence and regularity of critical points. The results are independent of the MLP activation function $\sigma$ and the vector field $f$ of the neural ODE. For non-augmented architectures and degenerate neural ODEs, we can directly infer that they cannot have the universal embedding property\index{Universal Embedding}, as not all function classes $(\mathcal{C}1)^k(\mathcal{X},\mathbb{R})$, $(\mathcal{C}2)^k(\mathcal{X},\mathbb{R})$ and $(\mathcal{C}3)^k(\mathcal{X},\mathbb{R})$ can be represented. With a perturbation result, it follows that non-augmented MLPs and neural ODEs also cannot have the universal approximation property\index{Universal Approximation} \cite{KuehnKuntz2024}. The result of non-universal approximation directly transfers to architectures with multi-dimensional output $q>1$, as universal approximation needs to hold for every single component of the considered network. In the augmented case, our results agree with the existing literature, as established universal approximation theorems for MLPs and neural ODEs use augmented architectures \cite{Hornik1989,  Kidger2022, KuehnKuntz2023, Pinkus1999}. In the case of MLPs with a bottleneck, the results in~\cite{KuehnKuntz2024} show a more detailed classification depending on the location of the bottleneck in the underlying graph. Furthermore, it is discussed why the results in the bottleneck/degenerate case are also comparable. The main reason is that neural ODEs cannot have a bottleneck in the hidden layer, as the dimension of the underlying initial value problem is constant over time. 

\begin{figure}
	\begin{tabular*}{\textwidth}{p{2.2cm}|p{6.0cm}|p{6.0cm}}
		& \hspace{10mm}\textbf{Multilayer Perceptron} & \hspace{17mm} \textbf{Neural ODE} \\[3pt]
		\hline && \\[-5pt]
		\textbf{Non-Augmented}
		&
		Every $\Phi_\theta \in \Xi^k_\textup{N}(\mathcal{X},\mathbb{R})$, $k\geq 1$, is of class $(\mathcal{C}1)^k(\mathcal{X},\mathbb{R})$.
		& 
		Every $\Phi_\theta \in \textup{NODE}^k_\textup{N}(\mathcal{X},\mathbb{R})$, $k\geq 1$, is of class $(\mathcal{C}1)^k(\mathcal{X},\mathbb{R})$. \\[5pt]
		\hline && \\[-5pt]
		\textbf{Augmented}
		&
		Every $\Phi_\theta \in \Xi^k_\textup{A}(\mathcal{X},\mathbb{R})$, $k \geq 2$, is, for all sets of weights apart from a zero set w.r.t.\ the Lebesgue measure in the weight space $\mathbb{W}$, of class $(\mathcal{C}1)^k(\mathcal{X},\mathbb{R})$ or $(\mathcal{C}2)^k(\mathcal{X},\mathbb{R})$. 
		& 
		Every $\Phi_\theta\in \textup{NODE}^k_\textup{A}(\mathcal{X},\mathbb{R})$, $k \geq 2$, is, for all sets of weights apart from a zero set w.r.t.\ the Lebesgue measure in the weight space $\mathbb{V}$, of class $(\mathcal{C}1)^k(\mathcal{X},\mathbb{R})$ or $(\mathcal{C}2)^k(\mathcal{X},\mathbb{R})$. \\[5pt]
		\hline && \\[-5pt]
		\textbf{Bottleneck / Degenerate} 
		&
		The map $\Phi_\theta \in \Xi^k_\textup{B}(\mathcal{X},\mathbb{R})$, $k\geq 2$, can be of all classes $(\mathcal{C}1)^k(\mathcal{X},\mathbb{R})$, $(\mathcal{C}2)^k(\mathcal{X},\mathbb{R})$ or $(\mathcal{C}3)^k(\mathcal{X},\mathbb{R})$. 
		& 
		Every $\Phi_\theta \in \textup{NODE}^k_\textup{D}(\mathcal{X},\mathbb{R})$, \mbox{$k\geq 2$,} is either of class $(\mathcal{C}1)^k(\mathcal{X},\mathbb{R})$ or $(\mathcal{C}3)^k(\mathcal{X},\mathbb{R})$.
	\end{tabular*}
	\vspace{3mm}
	\captionof{table}{Classification of scalar MLPs and scalar neural ODEs of Theorem \ref{cekk_th:classification}.} \label{cekk_tab:results}
\end{figure}

\subsection{Influence of the Memory Capacity on Universal Approximation}
\label{cekk_sec:delay}

In the last Section \ref{cekk_sec:morse}, we have seen that non-augmented neural ODEs cannot have the universal embedding property. If the phase space is sufficiently large, Theorem~\ref{cekk_th:universalembedding} shows that augmented neural ODEs have the universal embedding property. Naturally, the question arises if universal approximation can also be achieved not by increasing the width or the dimension of the phase space, but by adding memory \index{Memory of a Neural Network} to the network architecture. For DenseResNets with update rule~\eqref{cekk_eq:DenseNet}, memory is introduced by inter-layer connections, and for neural DDEs by the delay parameter~$\tau$ of the DDE~\eqref{cekk_eq:NeuralDDE}. \index{Neural Delay Differential Equation (Neural DDE)} In this section, we focus on neural DDEs as in Definition \ref{cekk_def:NDDE} and show, that depending on the memory capacity \index{Memory Capacity of a Neural DDE} $K\tau$, given by the product of the Lipschitz constant $K$ and the delay $\tau$, the class $\text{NDDE}^k_\tau(\mathcal{X},\R^q)$ has or does not have the universal approximation property with respect to the space $C^j(\mathcal{X},\R^q)$.

\begin{theorem}[Memory Dependence in Neural DDEs \cite{KuehnKuntz2025}] \label{cekk_th:memory}
	\index{Universal Approximation} Let $\Phi_\theta \in \textup{NDDE}_{\tau}^k(\mathcal{X},\R^q)$, $\mathcal{X} \subset \R^d$ open, $k\geq 1$, be a neural DDE with delay $\tau\geq0$. Let the underlying vector field $F$ be globally Lipschitz continuous with respect to the second variable with Lipschitz constant $K$. Fix some constants $w, \tilde w \in (0,\infty)$ and $A \geq 0$. Then the following assertions hold:
	\begin{enumerate}[label=(\alph*), font=\normalfont]
		\item If $m = d$, there exists a continuous function $\tau_0((0,\infty),[0,T])$ with $K\tau_0(K)e<1$ for $K\in(0,\infty)$, such that if $\tau\in[0,\tau_0(K))$, the class of neural DDEs $\textup{NDDE}^k_{\tau,\textup{N}}(\mathcal{X},\R^q)$ with $\vert\vert{F(t,0)}\vert\vert_\infty \leq A$ for all $t\in [0,T]$ and weight matrices $W, \widetilde{W}$ fulfilling $\vert\vert W \vert\vert_\infty \leq w$, $\vert\vert\widetilde W\vert\vert_\infty \leq \tilde w$, does not have the universal approximation property with respect to the space $C^j(\mathcal{X},\R^q)$, $j\geq 0$.
		\item Let $\Psi\in C^k_b(\mathcal{X},\R^q)$, $k\geq 1$ with global Lipschitz constant $K_\Psi$. Then if
		\begin{equation*}
			K\tau \geq 2\left(1+\frac{K_\Psi}{w \tilde w}\right),
		\end{equation*}
		and $m = \max\{d,q\}$, the map $\Psi$ can be represented by a neural DDE $\Phi_\theta\in\textup{NDDE}^k_{\tau}(\mathcal{X},\R^q)$ with weight matrices $W, \widetilde{W}$ with $\vert\vert{W}\vert\vert_\infty = w$, $\vert\vert\widetilde W\vert\vert_\infty = \tilde w$, delay $\tau$ and global Lipschitz constant $K$ in the second variable.
	\end{enumerate}
\end{theorem}

From part (a), it follows by deleting arbitrary components of the vector field that for $m \leq d$, neural DDEs cannot have the universal approximation property. The statement from part (b) generalizes to all $m\geq \max\{d,q\}$ by adding additional zero components to the vector field. For $m \geq d+q$, it follows analogously to Theorem~\ref{cekk_th:universalembedding}, that also neural DDEs $\Phi_\theta\in\textup{NDDE}^k_{\tau,\textup{A}}(\mathcal{X},\R^q)$ have the universal embedding property \index{Universal Embedding} with respect to the space $C^k(\mathcal{X},\R^q)$, $\mathcal{X} \subset \R^d$ open \cite{KuehnKuntz2025}. All results obtained for the memory dependence of neural DDEs are visualized in Figure~\ref{cekk_fig:NDDEtheorem}. In the case that the output dimension is not larger than the input dimension, i.e., $d\geq q$, it holds $\max\{d,q\} = d$, such that for $m = d$, both results of Theorem~\ref{cekk_th:memory}~(a) and~(b) apply. In summary, Theorem~\ref{cekk_th:memory} provides a mathematically precise statement regarding the role of memory in network architectures, i.e., how increasing memory enables universal approximation and universal embedding.

\begin{figure}[h]
	\centering
	\vspace{3mm}
	\begin{overpic}[scale = 0.85,,tics=10]
		{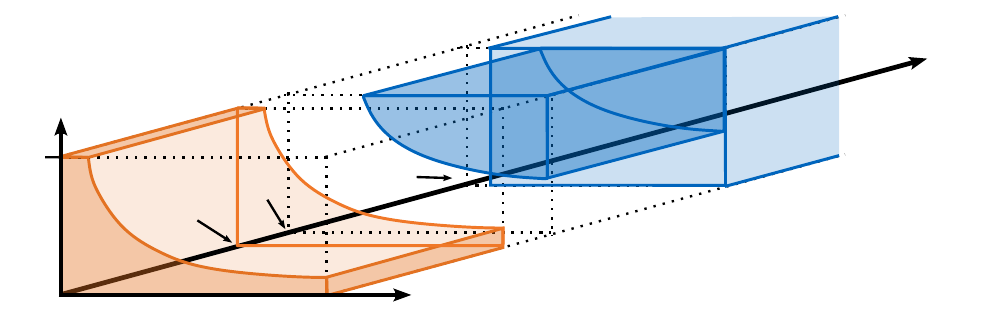}
		\put(42,2){$K$}
		\put(5,22){$\tau$}
		\put(1,16){$T$}
		\put(94,26){$m$}
		\put(17.5,10){$d$}
		\put(20.8,14){$\max\{d,q\}$}
		\put(34.5,14){$d+q$}
	\end{overpic}
	\caption{Joint visualization of the results of Theorem~\ref{cekk_th:universalembedding} and Theorem~\ref{cekk_th:memory} for the memory dependence of neural DDEs in the three-dimensional parameter space $(K,\tau,m)\in[0,\infty)\times[0,T]\times \mathbb{N}$ of Lipschitz constant $K$, delay $\tau$, and dimension $m$ of the vector field of the DDE. Parameter regions, where no universal approximation is possible, are visualized in orange, and parameter regions with the universal embedding property are visualized in blue. The figure is adapted from \cite[Figure 3.4]{KuehnKuntz2025}.}
	\label{cekk_fig:NDDEtheorem}
\end{figure}

\section{Optimization Algorithms as Dynamical Systems}\label{cekk_sec:opti}

In the last section we have shown, how dynamical systems can be used to study dynamics on fixed networks. In the following, we will discuss how the process of training a neural network can be interpreted as a dynamical system. In particular, we focus on characterizing the dynamical stability of minima and on the influence of stability on generalization. In Section \ref{cekk_subs:opti} we introduce the optimization tasks and the distinction between overdetermined and overparameterized problems. In Section \ref{cekk_subs:GDStab} we introduce a dynamical system viewpoint for gradient descent and study stability for overdetermined problems. In Section~\ref{cekk_subs:GDOpar}, we generalize this analysis to the overparameterized setting and discuss the edge of stability phenomenon and possible explanations for implicit bias. In Section~\ref{cekk_subs:RDS}, we introduce a random dynamical system framework for stochastic gradient descent and briefly discuss the overdetermined setting. Finally, in Section~\ref{cekk_sec:SGDOver}, we extend our stability analysis for the overparameterized setting to stochastic gradient descent.

\subsection{Optimization Tasks in Machine Learning} \label{cekk_subs:opti}
The goal of training a neural network $\Phi: \R^D \times \R^d \to \R^q,\,(\theta, x) \mapsto \Phi_\theta(x) = \Phi(\theta,x)$ (cf.~\eqref{cekk_eq:input_output_map}) is to find parameters $\theta$ such that $\Phi_\theta \approx \Psi$ in an appropriate sense, where $\Psi$ is the ground-truth function. In practice, the ground-truth function $\Psi$ is not known, and instead a suitable parameter $\theta$ has to be inferred from a set of $N$ training examples $(x_i, y_i)_{i \in {1,\dots, N}}$ with $y_i = \Psi(x_i)$ for every $i$. For this purpose, we introduce a function $\ell: \R^q \times \R^q \to [0, \infty), \, (y',y) \mapsto \ell(y',y)$ that measures how far a prediction~$y'$ is from the ground truth answer $y$. Of course, such a distance function should satisfy $\ell(y',y) = 0$ if and only if $y'=y$. We can then define the loss function\index{Loss Function} $\mathcal L: \mathbb R^D \to [0, \infty)$ by
\begin{equation}\label{cekk_eq:EmpLoss}
	\mathcal L(\theta) \coloneqq \frac{1}{N} \sum_{i = 1}^N \ell(\Phi(\theta, x_i), y_i).
\end{equation}
The value $\mathcal L(\theta)$, which is also called empirical error\index{Empirical Error}, measures how closely the neural network $\Phi_\theta$ approximates the ground truth function $\Psi$ in the training data set $\{x_1, \dots, x_N\}$.
For the regression problem we consider here, a natural choice for $\ell$ is given by $\ell(y',y) \coloneqq \frac{1}{2} \|y'-y\|_2^2$, where $\|\cdot\|_2$ denotes the Euclidean norm. In that case, the loss function $\mathcal L$ is the so-called mean-square error (MSE)\index{Mean-Square Error (MSE)}.

Instead of searching for a parameter $\theta$ for which $\Phi_\theta \approx \Psi$ directly, we seek a $\theta$ for which $\mathcal L(\theta) \approx 0$, implying that $\Phi_\theta(x) \approx \Psi(x)$, at least for all $x \in \{x_1, \dots, x_N\}$.
Ideally, we want to find a global minimum $\theta^*$ of the loss function $\mathcal L$, i.e.
\begin{equation} \label{cekk_eq:optiTask}
	\theta^* \in \operatorname{argmin}_{\theta \in \R^D} \mathcal L(\theta).
\end{equation}
When studying such an optimization task, it is crucial to differentiate whether the problem is {overdetermined} or {overparameterized}. An optimization task of the form~\eqref{cekk_eq:optiTask} is called overdetermined\index{Overdetermined Problem} if $D < qN$ and overparameterized\index{Overparameterized Problem} if $D > qN$. According to the famous \emph{double-descent curve}\index{Double-Descent Curve} \cite[Figure 1]{Belkin19}, both overdetemined and overparameterized models are able to perform well, while the ``middle-ground" $D \approx qN$ tends to perform poorly. The main difference between the two regimes can be seen by asking about the existence of a parameter $\theta$ with $\mathcal L(\theta) = 0$. By \eqref{cekk_eq:EmpLoss}, such a parameter $\theta$ has to interpolate the training data, i.e., it has to satisfy
\begin{equation}\label{cekk_eq:Interpolation}
	\Phi(\theta,x_i) = y_i, \quad \text{for all } i \in [N].
\end{equation}
In the overdetermined case $D<qN$, the number of equations in \eqref{cekk_eq:Interpolation} exceeds the dimension of $\theta$, and generically there is no $\theta \in \mathbb R^D$ with $\mathcal L(\theta) = 0$. Furthermore, the global minima of $\mathcal L$ are generically unique. Consequently, the main question of interest for the training process is whether the algorithm converges towards that global minimum.

In the overparameterized case $D>qN$, on the other hand, the dimension of $\theta$ exceeds the number of equations in \eqref{cekk_eq:Interpolation} and it is possible to have an infinite number of solutions to $\mathcal L(\theta) = 0$. In particular, the following statement follows from Sard's Theorem (see e.g.~\cite{Milnor65}), assuming $D>qN$:\index{Interpolation Solution}
\begin{theorem}[Set of Interpolation Solutions as Embedded Submanifold~\cite{Cooper21}]\label{cekk_theo:Cooper}
	For every continuously differentiable network function $\Phi: \R^D \times \R^d \to \R^q$ and every training data set $((x_1,y_1), \dots, (x_N,y_N)) \in (\R^d \times \R^q)^N$ outside a Lebesgue-null set (which may depend on $\Phi$), the set of interpolation solutions $\mathcal M \coloneqq \{\theta \in \R^D: \mathcal L(\theta) = 0 \}$ is either empty or a $(D-qN)$-dimensional embedded submanifold of~$\R^D$.
\end{theorem}
If the network function is sufficiently expressive, we expect that interpolating parameters exist, which means that $\mathcal M$ is a \emph{non-empty} $(D-qN)$-dimensional manifold. In that case, every $\theta \in \mathcal M$ is a global minimum of $\mathcal L$, and \eqref{cekk_eq:optiTask} does not define~$\theta^*$ uniquely. Consequently, it is not only of interest to ask whether an optimization algorithm converges to a global minimum, but also to which one. The latter question is of particular importance to understand implicit bias \index{Implicit Bias}. Here, implicit bias refers to the tendency of an optimization algorithm to find global minima $\theta^*$ that generalize\index{Generalization} well, meaning that $\Phi_{\theta^*}(x) \approx \Psi(x)$ even for $x \notin \{x_1, \dots, x_N\}$. 

\begin{figure}
	\centering
	\vspace{3mm}
	\includegraphics[width=\linewidth]{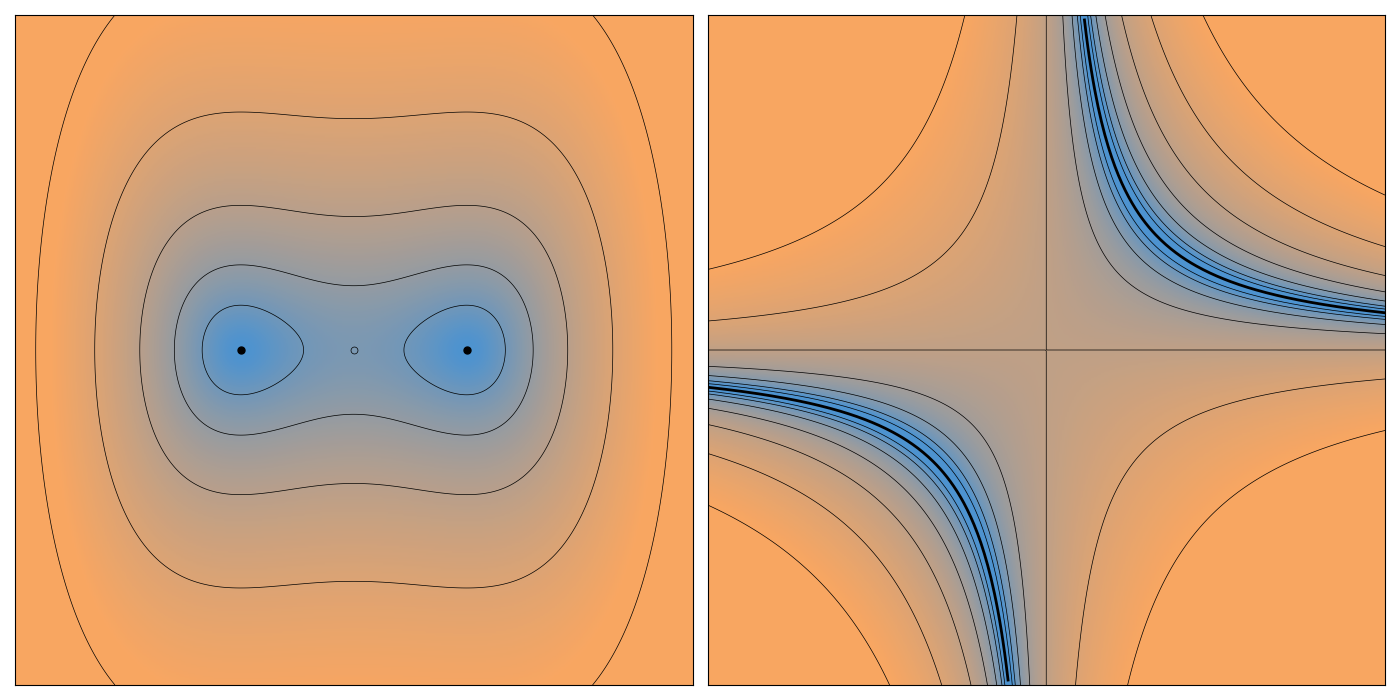}
	\caption{Typical geometries of the loss landscape in the overdetermined (left) and the overparameterized regime (right) for a two dimensional parameter space. The loss $\mathcal L$ is lower in blue regions and higher in orange regions. The thin black lines show level sets of $\mathcal L$. In the left plot, the two black dots indicate local minima and the hollow dot indicates a saddle point. In the right plot, the manifold $\mathcal M$ (cf.~Theorem~\ref{cekk_theo:Cooper}) of interpolation solutions is shown as a thick black line.}
	\label{cekk_fig:losslandscape}
\end{figure}

\subsection{Gradient Descent and Milnor Stability}\label{cekk_subs:GDStab}

The simplest algorithm that can be used to train neural networks is known as gradient descent (GD)\index{Gradient Descent (GD)}. After choosing some initial parameters $\theta_0$ deterministically, e.g. \mbox{$\theta_0=0$}, or randomly, e.g.~$\theta_0 \sim \mathcal N(0,\Sigma)$, the parameters are successively updated according to the recursion
\begin{equation}\label{cekk_eq:GD}
	\theta_{n+1} \coloneqq \varphi(\theta_n) \coloneqq \theta_n - \eta \nabla \mathcal L(\theta_n).
\end{equation}
Here $\eta$ is a positive real number called learning rate\index{Learning Rate} and $\nabla \mathcal L$ denotes the gradient of the loss function~$\mathcal L$. As motivated above, we are interested in the behavior of~$\theta_n$ as $n\to \infty$. The recursion \eqref{cekk_eq:GD} can be interpreted as a discrete time dynamical system\index{Dynamical System!Discrete Time Dynamical System} and it can be easily seen that we have $\theta_n = \varphi^n(\theta_0)$ for all $n\in \mathbb N$. This allows us to take a dynamical viewpoint, i.e., to shift the focus from considering a single trajectory to considering the geometry of the map $\varphi$ across the parameter space.
If $\theta^* \in \mathbb R^D$ is a local minimum of $\mathcal L$ then $\nabla \mathcal L(\theta^*) = 0$ and thus $\theta^*$ is a fixed point of $\varphi$, i.e., we have $\varphi(\theta^*) = \theta^*$. Whether the gradient descent algorithm can actually converge to the local minimum $\theta^*$ depends on its dynamical stability. 

Consider, for now, the overdetermined setting $D<qN$.\index{Overdetermined Problem} We introduce the following definitions of stability for isolated fixed points, i.e., fixed points with a non-empty neighbourhood that does not contain another fixed point.
\begin{definition}[Milnor Stability of an Isolated Fixed Point]\label{cekk_def:Milnor}~
	\begin{enumerate}
		\item[(i)]\index{Dynamical Stability!for Isolated Minima}An isolated fixed point $\theta^* = \varphi(\theta^*)$ is called \emph{Milnor stable} if the set of possible initializations $\theta_0\in \R^D$ for which the sequence of gradient descent iterates $\theta_n$ converges to $\theta^*$ has a positive Lebesgue measure, i.e., if
		$$\operatorname{Vol}\left\{\theta \in \R^D : \lim_{n \to \infty} \varphi^n(\theta) = \theta^*\right\}>0.$$
		\item[(ii)]An isolated fixed point $\theta^* = \varphi(\theta^*)$ is called \emph{locally Milnor stable} if for every open neighborhood $U\subset \R^D$ of $\theta^*$ the set of possible initializations $\theta_0\in U$, for which the sequence of gradient descent iterates $\theta_n$ stays in $U$ forever and converges to $\theta^*$, has a positive Lebesgue measure, i.e., if
		$$\operatorname{Vol}\left\{\theta \in U : \varphi^n(\theta) \in U,~\forall\,n \in \mathbb N \text{ and }\lim_{n \to \infty} \varphi^n(\theta) = \theta^*\right\}>0.$$
	\end{enumerate}
\end{definition}
The name Milnor stability is motivated by the fact that a fixed point $\theta^*$ is Milnor stable if and only if the singleton set $\{\theta^*\}$ is a Milnor attractor \cite{Milnor85}\index{Milnor Attractor}. Milnor stability is an interesting concept for optimization tasks due to the following fact, which is a direct consequence of the definition.
\begin{proposition} \label{cekk_prop:OdetInterpr}
	Let $\theta^*$ be a local minimum.
	For a probability measure $\nu$ on the parameter space $\R^D$, let $\theta_0$ be a $\R^D$-valued random variable with law $\nu$ and let $\theta^n = \varphi^n(\theta_0)$ be the iterations of gradient descent. Suppose $\nu$ is equivalent to the Lebesgue measure, i.e.,
	\begin{equation}\label{cekk_def:equi}
		\nu(A) = 0 \Leftrightarrow \operatorname{Vol}(A) = 0, \text{ for each Borel set $A\subset \R^D$}.
	\end{equation}
	Then $\nu(\lim_{n \to \infty} \theta^n =\theta^*)>0$, i.e., $\theta^*$ is reached with positive probability, if and only if $\theta^*$ is Milnor stable.
\end{proposition}

Note that the notion of local Milnor stability is stronger than the notion of Milnor stability. Nevertheless, if the map $\varphi$ is non-singular, meaning that the preimages of Lebesgue-null sets are Lebesgue-null sets, the two definitions are equivalent.

\begin{proposition}\label{cekk_prop:LocGlobEqui}
	Let $\theta^*$ be a local minimum of $\mathcal L$ and suppose that the map $\varphi$ is non-singular. Then $\theta^*$ is locally Milnor stable if and only if it is Milnor stable.
\end{proposition}
\begin{proof}
	It follows directly from Definition \ref{cekk_def:Milnor} that every locally Milnor stable point is Milnor stable. It remains to show the converse. Suppose that $\varphi$ is non-singular and that $\theta^*$ is not locally Milnor stable and let $U\subset \R^D$ be a neighborhood of $\theta^*$ for which the set
	$$N\coloneqq \left\{\theta \in U : \varphi^n(\theta) \in U,~\forall\,n \in \mathbb N, \text{ and }\lim_{n \to \infty} \varphi^n(\theta) = \theta^*\right\}$$
	is a Lebesgue null set. By repeated application of the non-singularity of $\varphi$ we get that
	$$\operatorname{Vol}\left(\varphi^{-m}(N)\right) = 0, ~\forall\, m \in \mathbb N.$$
	By the definition of convergence, we have
	$$\left\{\theta \in \R^D : \lim_{n \to \infty} \varphi^n(\theta) = \theta^*\right\} = \bigcup_{m = 1}^\infty \varphi^{-m}(N)$$
	and thus 
	$$\operatorname{Vol}\left\{\theta \in \R^D : \lim_{n \to \infty} \varphi^n(\theta) = \theta^*\right\}=0$$
	implying that $\theta^*$ is not Milnor stable.
\end{proof}
It has been shown in \cite[Theorem 12]{Cruaciun24} that the gradient descent update map $\varphi$ is, indeed, generically non-singular, making the two definitions virtually equivalent for practical purposes. Although local Milnor stability is strictly weaker than the classical notion of asymptotic stability, it, and more importantly, the lack of it, can be shown using the same techniques.

Let $\theta^*$ be a strict local minimum, i.e., the Hessian $\operatorname{Hess} \mathcal L(\theta^*)$ is positive definite. In particular, $\theta^*$ is an isolated local minimum whose stability is determined by the linearization of $\varphi$ around $\theta^*$. 
\begin{theorem}[Spectral Stability and Milnor Stability]
	For a strict local minimum~$\theta^*$, we let $\rho(\mathrm D \varphi(\theta^*))$ denote the spectral radius of the Jacobian of $\varphi$ in $\theta^*$. Then $\theta^*$ is locally Milnor stable if $\rho(\mathrm D \varphi(\theta^*))< 1$, and only if $\rho(\mathrm D \varphi(\theta^*))\leq 1$.
\end{theorem}
\begin{proof}[Sketch of proof]
	In the case $\rho(\mathrm D \varphi(\theta^*))< 1$, the local stability follows from the stable manifold theorem \cite[Section 6.2]{Katok1995}. In the case $\rho(\mathrm D \varphi(\theta^*))> 1$, the lack of local stability follows from the center-unstable manifold theorem \cite[Section 6.2]{Katok1995}.
\end{proof}

Since $\mathrm D \varphi(\theta^*) = \mathds 1_D - \eta \operatorname{Hess} \mathcal L(\theta^*)$ (cf.~\eqref{cekk_eq:GD}), with $\operatorname{Hess} \mathcal L (\theta^*)$ positive definite, the spectral radius of $\mathrm D \varphi(\theta^*)$ is less than 1 if and only if 
\begin{equation}\label{cekk_ineq:StabCond}
	2> \eta\rho(\operatorname{Hess} \mathcal L(\theta^*)) = \eta \|\operatorname{Hess} \mathcal L(\theta^*)\|,
\end{equation}
where the equality uses the fact that the Hessian is symmetric and the matrix norm on the right-hand side is the operator norm derived from the Euclidean norm on $\mathbb R^D$.

\begin{remark}
	The linear approximation 
	$$\varphi(x) = x^* + \mathrm D \varphi(x^*) (x-x^*) + o (|x-x^*|)$$
	of $\varphi$ around a local minimum $\theta^*$ is the dynamical analogue of the quadratic approximation\index{Quadratic Approximation}
	$$\mathcal L (\theta) = \mathcal L(\theta^*) + \langle\theta-\theta^*, \operatorname{Hess} \mathcal L (\theta^*) (\theta-\theta^*)\rangle + o\left(|\theta-\theta^*|^2\right)$$
	of the loss function $\mathcal L$ which is more commonly used to study the stability of local minima. The stability condition \eqref{cekk_ineq:StabCond} can be derived equally quickly by considering the quadratic approximation of the loss function, in which case the convergence result is commonly referred to as the Descent Lemma\index{Descent Lemma} \cite[Proposition 1.2.3]{Bertsekas97}. We advocate for the dynamical viewpoint since it is more suitable for extending the following stability analysis to stochastic optimization algorithms (cf.~Section \ref{cekk_sec:SGDOver} below).
\end{remark}

\subsection{Stability of Gradient Descent in the Overparameterized Setting}\textbf{\label{cekk_subs:GDOpar}}\index{Overparameterized Problem}
Shifting our attention to the overparameterized setting $D>qN$, strict local minima no longer exist. Indeed, it can be easily seen that the Hessian of the loss function~$\mathcal L$ has rank at most $qN$ and will thus always have a non-trivial kernel if $D>qN.$ Instead, as a standing assumption, we assume that the set of interpolation parameters $\mathcal M \coloneqq \{\theta \in \R^D: \mathcal L(\theta) = 0 \}$ is non-empty and forms a $(D-qN)$-dimensional embedded submanifold of $\R^D$ (cf.~Theorem \ref{cekk_theo:Cooper}). Although local minima might still exist, we will focus solely on these global minimizers in the overparameterized setting. As before, every $\theta^* \in \mathcal M$ is a fixed point of $\varphi$. Unlike the overdetermined setting, we should no longer expect to be able to show that $\mathbb P(\theta_n \to \theta^*) > 0$ for any fixed $\theta^* \in \mathcal M$ since $\mathcal M$ is uncountable. Instead, we seek to characterize for which minimizers $\theta^*$ the algorithm has a positive probability of converging to a minimum close to $\theta^*$. This leads us to adapting the notion of stability.
\begin{definition}[Stability of GD in the Overparameterized Setting]~\label{cekk_def:StabGDOpar}
	\begin{enumerate}
		\item[(i)]\index{Dynamical Stability!for Overparameterized Problems} A global minimum $\theta^*\in \mathcal M$ is called \emph{stable} if for every neighborhood $V\subset \mathcal M$ of $\theta^*$ which is open in the submanifold topology of $\mathcal M$, the set of possible initializations $\theta_0$ for which gradient descent converges to some global minimum $\theta' \in V$ has positive Lebesgue measure, i.e.
		$$\operatorname{Vol}\left\{\theta \in \R^D :\exists\, \theta' \in V, \text{ s.t. } \lim_{n \to \infty} \varphi^n(\theta) = \theta'\right\}>0.$$
		\item[(ii)] A global minimum $\theta^*\in \mathcal M$ is called \emph{locally stable} if there exists an open neighborhood $U$ of $\theta^*$, such that for every neighborhood $V\subset \mathcal M$ of $\theta^*$ which is open in the manifold topology of $\mathcal M$, the set of possible initializations $\theta_0$ for which gradient descent stays in $U$ forever and converges to some global minimum $\theta' \in V$ has positive Lebesgue measure, i.e.
		$$\operatorname{Vol}\left\{\theta \in \R^D :\forall\, n \in \mathbb N, ~\varphi^n(\theta) \in U \text{ and }\exists\, \theta' \in V, \text{ s.t. } \lim_{n \to \infty} \varphi^n(\theta) = \theta'\right\}>0.$$
	\end{enumerate}
\end{definition}
Analogously to Proposition \ref{cekk_prop:OdetInterpr}, this notion of stability has the following interpretation, which, again, follows directly from the definition.
\begin{proposition} \label{cekk_prop:OparInterpr}
	Let $\theta^* \in \mathcal M$ be a global minimum.
	For a probability measure $\nu$ on the parameter space $\R^D$, let $\theta_0$ be a $\R^D$-valued random variable with law $\nu$ and let $\theta^n = \varphi^n(\theta_0)$ be the iterations of gradient descent. We define a $(\mathcal M \cup \{\emptyset\})$-valued random variable $\theta_{\lim}$ by
	\begin{equation}\label{cekk_eq:thetaLim}
		\theta_{\lim} \coloneqq \begin{cases} \lim_{n \to \infty} \theta_n,&\text{if }\exists\, \theta \in \mathcal M, ~\lim_{n \to \infty} \theta_n = \theta,\\
			\emptyset,& \text{otherwise}.
		\end{cases}
	\end{equation}
	Here $\emptyset$ should be seen as an exception state accounting for the possibility that gradient descent does not converge to $\mathcal M$.
	Suppose $\nu$ is equivalent to the Lebesgue measure (cf.~\eqref{cekk_def:equi}) and define\index{Support!of a Random Variable}
	$$\operatorname{supp}(\theta_{\lim}) = \operatorname{supp}(\operatorname{law}(\theta_{\lim})) = \{\theta \in \mathcal M: \forall\, U \in \mathcal U_{\theta, \mathcal M},~ \nu(\{\theta_0 : \theta_{\lim} \in U\}) > 0\},$$
	where $\mathcal U_{\theta, \mathcal M}$ denotes the family of neighborhoods of $\theta$ which are open in the manifold topology of $\mathcal M$.
	
	Then $\theta^* \in \operatorname{supp}(\theta_{\lim})$ if and only if $\theta^*$ is stable in the sense of Definition \ref{cekk_def:StabGDOpar}(i).
\end{proposition}
If the map $\varphi$ is non-singular, it can be shown that the notion of stability in Definition \ref{cekk_def:StabGDOpar}(i) and the notion of local stability in Definition \ref{cekk_def:StabGDOpar}(ii) are equivalent, analogously to the proof of Proposition~\ref{cekk_prop:LocGlobEqui}. Again, we aim to characterize the minimizers that are locally stable.

For a global minimum $\theta^* \in \mathcal M$, let $\mathcal T(\theta^*)$ denote the tangent space and $\mathcal N(\theta^*) = \mathcal T(\theta^*)^\bot$ denote the normal space at $\theta^*$ to the manifold $\mathcal M$. It can be easily checked that 
\begin{align*}
	\mathcal T(\theta^*) &= \operatorname{ker} (\operatorname{Hess} \mathcal L (\theta^*) )&&\text{and}&\mathcal N(\theta^*) &= \operatorname{im} (\operatorname{Hess} \mathcal L (\theta^*)).
\end{align*}
Consequently, the Jacobian $\mathrm D \varphi(\theta^*) = \mathds 1_D - \eta \operatorname{Hess} \mathcal L (\theta^*)$ respects the splitting $\R^D = \mathcal T(\theta^*) \otimes \mathcal N(\theta^*)$ with $\mathrm D \varphi(\theta^*) \vartheta = \vartheta$ for every $\vartheta \in \mathcal T(\theta^*)$. In particular, $1$ is an eigenvalue of $\mathrm D \varphi(\theta^*)$ with geometric multiplicity at least $D-qN$. Thus, the stability condition $\rho(\mathrm D \varphi(\theta^*)) < 1$ from the previous section is never satisfied, and instead we should consider the restriction of $\mathrm D \varphi(\theta^*)$ to the normal space $\mathcal N(\theta^*)$, which defines a linear automorphism $\mathrm D \varphi(\theta^*)|_{\mathcal N(\theta^*)} : \mathcal N(\theta^*) \to \mathcal N(\theta^*)$. This linear map determines the transverse stability of $\mathcal M$ in $\theta^*$.\index{Dynamical Stability!Transverse Stability}
\begin{theorem}[Spectral Stability and Local Stability of Global Minima]
	A global minimum $\theta^*$ is locally stable in the sense of Definition \ref{cekk_def:StabGDOpar}(ii) if $\rho(\mathrm D \varphi(\theta^*)|_{\mathcal N(\theta^*)})< 1$, and only if $\rho(\mathrm D \varphi(\theta^*)|_{\mathcal N(\theta^*)})\leq 1$.
\end{theorem}
\begin{proof}[Sketch of proof]
	In the case $\rho(\mathrm D \varphi(\theta^*)|_{\mathcal N(\theta^*)})< 1$, the local stability follows from the stable manifold theorem. In the case $\rho(\mathrm D \varphi(\theta^*)|_{\mathcal N(\theta^*)})> 1$, the lack of local stability follows from the center-unstable manifold theorem. For details, see \cite[Theorem A]{ChemnitzEngel24} or \cite[Section 2]{Ahn22}.
\end{proof}
Using the identity $\mathrm D \varphi(\theta^*)|_{\mathcal N(\theta^*)} = \mathds 1_D|_{\mathcal N(\theta^*)} - \eta \operatorname{Hess} \mathcal L(\theta^*)|_{\mathcal N(\theta^*)}$ and the fact that the operator $\mathcal L(\theta^*)|_{\mathcal N(\theta^*)}: \mathcal N(\theta^*) \to \mathcal N(\theta^*)$ is positive definite, we can see that the stability condition $\rho(\mathrm D \varphi(\theta^*)|_{\mathcal N(\theta^*)})< 1$ is equivalent to 
\begin{equation}\label{cekk_ineq:StabCond2}
	\|\operatorname{Hess} \mathcal L(\theta^*)|_{\mathcal N(\theta^*)}\| < \frac{2}{\eta}.
\end{equation}

\begin{figure}
	\centering
	\begin{overpic}[width=0.55\linewidth]{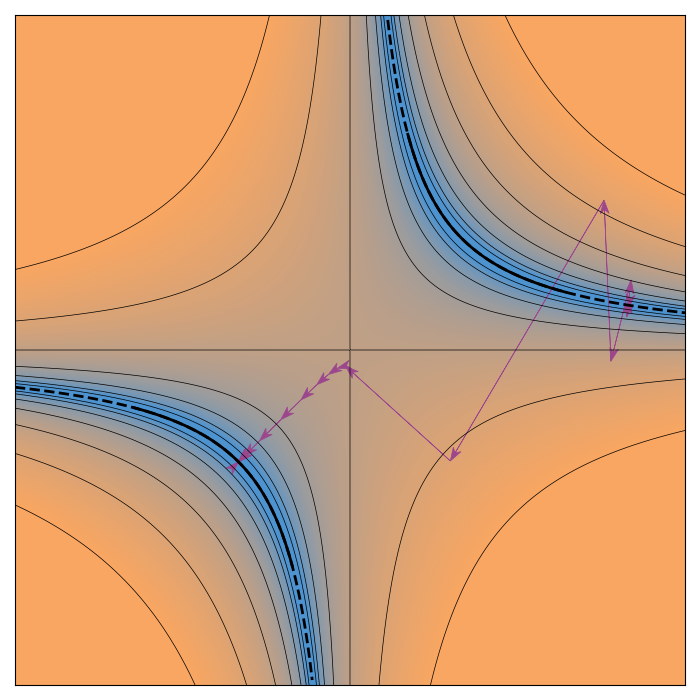}
		\put(47,-3){$\theta^{(1)}$}
		\put(-5,49){$\theta^{(2)}$}
	\end{overpic}
	\caption{Trajectory of gradient descent for an overparameterized optimization problem. The loss function is given by $\mathcal L(\theta) = (1-\theta^{(1)}\theta^{(2)})^2$ for $\theta = (\theta^{(1)}, \theta^{(2)}) \in \R^2$. The unstable sections of the manifold $\mathcal M = \{\theta^{(1)}\theta^{(2)} = 1\}$ of global minima are shown as a dashed line and the stable sections as a solid line. A trajectory of gradient descent with learning rate $\eta = \frac{1}{2}$ is shown in pink. It is initialized with $\theta_0 = (2.5, 0.41)$ close to a unstable region of $\mathcal M$. Gradient descent escapes from that region of $\mathcal M$ and converges to a point in the stable region instead.}
	\label{cekk_fig:overparGD}
\end{figure}

Although this is a direct analog of the stability condition \eqref{cekk_ineq:StabCond} for the overdetermined setting, the flavor is quite different in the overparameterized setting. The condition \eqref{cekk_ineq:StabCond} suggests that the learning rate $\eta$ should be sufficiently small so that some minimum of interest is stable under the algorithm. In contrast, the condition~\eqref{cekk_ineq:StabCond2} should be interpreted as a condition on the global minimum in question rather than as a condition on the learning rate. From this viewpoint, it states that the algorithm can only converge to global minima at which the loss function is sufficiently flat. Since flat minima\index{Flat Minimum} have, at least heuristically, been associated with good generalization\index{Generalization} \cite{Hochreiter97FlatMinima}, this leads to a possible explanation of implicit bias\index{Implicit Bias}. For the choice of the learning rate\index{Learning Rate}, this suggests that higher learning rates should lead to a stricter selection, a flatter minimum, and better generalization. This hypothesis, indeed, has empirical support \cite{Goyal17, Hoffer17}. 

Finally, we point the reader to the numerical study \cite{Cohen20}. Here, it is observed that in practically relevant scenarios, gradient descent converges to the so-called \emph{edge of stability}\index{Edge of Stability}, i.e., to a minimum $\theta^*$ for which 
$$ \|\operatorname{Hess} \mathcal L(\theta^*)|_{\mathcal N(\theta^*)}\| \approx \frac{2}{\eta}.$$
This indicates that the condition \eqref{cekk_ineq:StabCond2} is not just a mathematical technicality. Instead, the chosen learning rate directly influences the flatness of the minimum found by the gradient descent algorithm. In particular, it demonstrates that during the late stages of training the dynamics of gradient descent are qualitatively different from its continuous-time analogue, the gradient flow\index{Gradient Flow} $\dot \theta = - \nabla \mathcal L(\theta)$, under which every strict local minimum is stable. This can, for example, be seen from the fact that gradient descent does not reduce the loss monotonically during the later stages of training, as is demonstrated in \cite{Cohen20}. 

\subsection{Stochastic Gradient Descent as a Random Dynamical System} \label{cekk_subs:RDS}
Most algorithms employed in practice are variants of stochastic gradient descent (SGD)\index{Stochastic Gradient Descent}. Instead of optimizing $\mathcal L$ directly in every iteration, these algorithms pick for each step a random subset $\Xi$ called mini-batch and then optimize for the loss $\mathcal L_\Xi$ corresponding to that subset, which is given by
$$\mathcal L_\Xi(\theta) = \frac{1}{|\Xi|} \sum_{i \in \Xi} \ell(\Phi(\theta,x_i),y_i).$$
Put more precisely, the update rule \eqref{cekk_eq:GD} is replaced by 
$$\theta_{n+1} \coloneqq \theta_n - \eta \nabla \mathcal L_{\Xi_{n+1}}(\theta_n).$$
Here $\Xi_1, \Xi_2,\dots$ is a sequence of random independent, uniformly distributed $B$-element subsets of $[N]$, where $1\leq B<N$ is called the mini-batch size\index{Mini-Batch Size}.

While the stability condition \eqref{cekk_ineq:StabCond2} is widely accepted to be the correct condition for deterministic gradient descent, several different notions have been suggested for stochastic gradient descent \cite{Wu18, MaYing21, ChemnitzEngel24, Andreyev25}. Here, following \cite{ChemnitzEngel24}, we will extend our analysis for gradient descent to the stochastic case, leading to a stability condition in terms of the Lyapunov exponent of a random matrix product\index{Random Matrix Product}. For a detailed discussion on how this condition relates to other notions of stability proposed in the literature, see \cite[Appendix A]{ChemnitzEngel24}.

To extend the stability analysis for gradient descent to stochastic gradient descent, we first introduce a dynamical system framework for stochastic gradient descent. Let $(\Omega, \mathcal F, \mathbb P)$ be the probability space on which the random variables $(\Xi_i)$ are defined, so that for each $i\in \N$ the random variable
$\Xi_i: \Omega \to \{\Xi\subset [N] : |\Xi| = B \}$
is $\mathcal F$-measurable. Stochastic gradient descent defines a random dynamical system\index{Dynamical System!Random Dynamical System} \cite{Arnold98}, which can be constructed as follows. For each possible mini-batch $\Xi \subset [N]$, $|\Xi| = B$, we define a map $\varphi_\Xi: \R^D\to \R^D$ by
$$\varphi_\Xi(\theta) \coloneqq \theta - \nabla \mathcal L_{\Xi}(\theta).$$
For each $\omega \in \Omega$ and each $n \in \mathbb N_0$ we, furthermore, define a map $\varphi_\omega^{(n)}: \R^D \to \R^D$ by $\varphi_\omega^{(0)}(x) \coloneqq x$ and 
$$\varphi^{(n)}_\omega \coloneqq \varphi_{\Xi_n(\omega)} \circ \dots \circ \varphi_{\Xi_1(\omega)}.$$
Again, we have $\varphi^{(n)}_\omega(\theta_0) = \theta_n$.
Unlike the deterministic setting, we no longer have the relation $\varphi^{(n+m)}_\omega = \varphi^{(n)}_\omega \circ \varphi^{(m)}_\omega$. To formulate the analogous property for the random case, we first need to define a shift action on the probability space. Let $\varsigma: \Omega \to \Omega$ be a measurable map which preserves the probability measure $\mathbb P$, i.e., which satisfies 
$$\mathbb P(\varsigma^{-1}(E)) = \mathbb P(E),~\text{for each event } E\in \mathcal F,$$
and for which
$$(\Xi_1(\varsigma \omega), \Xi_2(\varsigma \omega), \dots) = (\Xi_2(\omega), \Xi_3(\omega), \dots).$$
With this notation, it can be checked that the maps $\varphi^{(n)}_\omega$ satisfy the cocycle property\index{Cocycle}
$$\varphi^{(n+m)}_\omega = \varphi^{(n)}_{\varsigma^m\omega} \circ \varphi^{(m)}_\omega,$$
turning the pair $(\varsigma, \varphi)$ into a random dynamical system.

Unlike gradient descent, local or even global minima $\theta^*$ are no longer necessarily fixed points for stochastic gradient descent. However, the interpolation parameters $\theta^* \in \mathcal M$, are almost-sure fixed points of the random dynamical system $\varphi_\omega^{(n)}$, i.e., we have $\varphi_\Xi(\theta^*) = \theta^*$ for every mini-batch $\Xi \subset [N]$, $|\Xi| = B$ and thus $\varphi_\omega^{(n)}(\theta^*)= \theta^*$ almost-surely for any $n \in \N$. Since $\mathcal M$ is typically empty in the overdetermined setting, we can extend the stability analysis to stochastic algorithms only in the overparameterized setting. We mention that in the overdetermined setting, it is still possible to study the convergence of the law of $\theta_n$ to a stationary distribution\index{Stationary Distribution} of the stochastic dynamics, rather than to a single parameter \cite{Gurbuzbalaban21}. Alternatively, one can consider stochastic gradient descent with a decaying learning rate\index{Learning Rate!Decaying Learning Rate}, which does allow for convergence (see e.g.~\cite{Benaim96, Fehrman20}). Both of these approaches are outside of our scope, and we restrict ourselves to the overparameterized setting $D>qN$ in the following.

\subsection{SGD for Overparameterized Problems}\label{cekk_sec:SGDOver}
Analogously to Section \ref{cekk_subs:GDOpar}, we again assume as a standing assumption that the set of interpolation parameters $\mathcal M \coloneqq \{\theta \in \R^D: \mathcal L(\theta) = 0 \}$ is non-empty and forms a $(D-qN)$-dimensional embedded submanifold of $\R^D$ (cf.~Theorem~\ref{cekk_theo:Cooper}). We generalize Definition \ref{cekk_def:StabGDOpar} to stochastic optimization algorithms in the following way.
\begin{definition}
	[Stability of SGD in the Overparameterized Setting]~\label{cekk_def:StabSGD}
	\begin{enumerate}\index{Dynamical Stability!for SGD}
		\item[(i)] A global minimum $\theta^*\in \mathcal M$ is called \emph{stable} if for every neighborhood $V\subset \mathcal M$ of $\theta^*$ which is open in the submanifold topology of $\mathcal M$, the set of possible initializations $\theta_0$ for which stochastic gradient descent converges with positive probability to some global minimum $\theta' \in V$ has positive Lebesgue measure, i.e.
		$$\operatorname{Vol}\left\{\theta \in \R^D :\mathbb P\left(\exists\, \theta' \in V, \text{ s.t. } \lim_{n \to \infty} \varphi^{(n)}_\omega(\theta) = \theta'\right) > 0\right\}>0.$$
		\item[(ii)] A global minimum $\theta^*\in \mathcal M$ is called \emph{locally stable} if there exists an open neighborhood $U$ of $\theta^*$, such that for every neighborhood $V\subset \mathcal M$ of $\theta^*$ which is open in the manifold topology of $\mathcal M$, the set of possible initializations $\theta_0$ for which stochastic gradient descent has a positive probability to stay in $U$ forever while converging to some global minimum $\theta' \in V$, has positive Lebesgue measure, i.e.,
		\begin{multline*}
			\operatorname{Vol}\bigg\{\theta \in \R^D : \mathbb P \left(\forall\, n \in \mathbb N, ~\varphi^n(\theta) \in U \text{ and }\exists\, \theta' \in V, \text{ s.t. } \lim_{n \to \infty} \varphi^n(\theta) = \theta'\right) > 0\bigg\}>0.
		\end{multline*}
	\end{enumerate}
\end{definition}

Again, this notion of stability can be interpreted in terms of the support of $\theta_{\lim}$. Suppose that $\theta_0$ is initialized randomly according to a probability measure $\nu$ on $\R^D$ which is equivalent to the Lebesgue measure (cf.~\eqref{cekk_def:equi}). Let $\theta_n = \varphi_\omega^{(n)}(\theta_0)$ be the iterates of stochastic gradient descent and define the random variable $\theta_{\lim}$ as in \eqref{cekk_eq:thetaLim}. This time, $\theta_{\lim}$ depends on the randomness of initialization and training. Accordingly, we now define $\operatorname{supp}(\theta_{\lim})$ as
$$\operatorname{supp}(\theta_{\lim}) \coloneqq \{\theta \in \mathcal M: \forall\, U \in \mathcal U_{\theta, \mathcal M},~ (\mathbb P \times \nu)(\{(\omega,\theta_0) : \theta_{\lim} \in U\}) > 0\}.$$

\begin{proposition}
	For a global minimum $\theta^* \in \mathcal M$, we have $\theta^* \in \operatorname{supp}(\theta_{\lim})$ if and only if $\theta^*$ is stable in the sense of Definition \ref{cekk_def:StabSGD}(ii).
\end{proposition}
If the function $\varphi_\Xi$ is non-singular for every mini-batch $\Xi \subset [N]$, the proof of Proposition \ref{cekk_prop:OdetInterpr} can be easily adapted to show that every stable minimum is also locally stable. Since $\varphi_\Xi$ is the gradient descent map for a reduced training data set, the result of \cite{Cruaciun24} still shows that this condition is satisfied generically. Similarly to our analysis for gradient descent, we can determine whether a global minimum $\theta \in \mathcal M$ is locally stable by linearizing the update maps. However, for stochastic gradient descent, we get a different Jacobian $\mathrm D \varphi_\Xi(\theta^*)$ for each mini-batch $\Xi$. Let $A_\Xi(\theta^*) = \mathrm D \varphi_\Xi(\theta^*)|_{\mathcal N(\theta^*)}: \mathcal N(\theta^*) \to \mathcal N(\theta^*)$
be the transverse part of the Jacobian of $\varphi_\Xi$. The correct analog of the leading eigenvalue of a single matrix to random matrix products is given by the Lyapunov exponent\index{Lyapunov Exponent}, the existence of which is guaranteed by the Furstenberg-Kesten Theorem\index{Furstenberg-Kesten Theorem}.
\begin{theorem}[Furstenberg-Kesten Theorem \cite{Furstenberg60}]
	Let $\theta^*\in \mathcal M$ be a global minimum. There exists a (deterministic) real number $\lambda(\theta^*)$ called the \emph{Lyapunov exponent}, such that for almost every $\omega \in \Omega$ we have
	$$\lambda(\theta^*) = \lim_{n \to \infty} \frac{1}{n} \log \left\|\mathrm D \left(\varphi_\omega^{(n)}\right)(\theta^*)\big|_{\mathcal N(\theta^*)}\right\| = \lim_{n \to \infty} \frac{1}{n} \log \|A_{\Xi_n(\omega)} \dots A_{\Xi_1(\omega)}\|.$$
\end{theorem}
The condition $\lambda(\theta^*) < 0$ is the correct generalization of the stability condition $\rho(\mathrm D \varphi(\theta^*)|_{\mathcal N(\theta^*)}) < 1$ for gradient descent to stochastic gradient descent. In fact, if $\rho(\mathrm D \varphi(\theta^*)|_{\mathcal N(\theta^*)}) < 1$, by Gelfand's formula, we have
\begin{align*}
	\lim_{n \to \infty} \frac{1}{n} \log \left\|\mathrm D \left(\varphi^{n}\right)(\theta^*)\big|_{\mathcal N(\theta^*)}\right\| &= \lim_{n \to \infty} \frac{1}{n} \log \left\|\left(\mathrm D \varphi(\theta^*)\big|_{\mathcal N(\theta^*)}\right)^n\right\| \\
	&= \log(\rho(\mathrm D \varphi(\theta^*)|_{\mathcal N(\theta^*)})) < 0,
\end{align*}
establishing the connection. Thus, it is reasonable to assume that $\theta^*$ is locally stable if $\lambda(\theta^*) < 0$ and only if $\lambda(\theta^*) \leq 0$. However, the proof is significantly more challenging than in the deterministic case due to the possibility of non-uniform hyperbolicity\index{Non-Uniform Hyperbolicity}. To tackle this hurdle, extra assumptions on the global minimum in question are required.\index{Regular Minimum}
\begin{definition}[Regular Minima]
	A global minimum $\theta^* \in \mathcal M$ is called regular if
	\begin{enumerate}
		\item[(i)] for each mini-batch $\Xi$ the operators $A_\Xi(\theta^*)$ and $\operatorname{Id}-A_\Xi(\theta^*)$ are invertible, and
		\item[(ii)] the semigroup $S(\theta^*)$ generated by $\{A_\Xi(\theta^*): \Xi \subset [N],~|\Xi| = B\}$ is strongly irreducible, i.e., there is no finite set of proper subspaces of $\mathcal N(\theta^*)$ which is invariant under $S(\theta^*)$.
	\end{enumerate}
\end{definition}

Note that a generic global minimum is regular. With this definition at hand, it has been shown in \cite{ChemnitzEngel24} that, for regular global minima $\theta^*$, the Lyapunov exponent $\lambda(\theta^*)$ does indeed determine the local stability of $\theta^*$.
\begin{theorem}[Local Stability of Regular Minima for SGD~{\cite[Theorem B]{ChemnitzEngel24}}]\index{Dynamical Stability!for SGD}\label{cekk_theo:SGD}
	Let $\theta^* \in \mathcal M$ be a regular global minimum. Then $\theta^*$ is locally stable if $\lambda(\theta^*) < 0$ and only if $\lambda(\theta^*) \leq 0$.
\end{theorem}

The proof of Theorem \ref{cekk_theo:SGD}, as presented in \cite{ChemnitzEngel24}, is inspired by a technique developed by Baxendale and Stroock \cite{BaxendaleStroock88} to study the stability of invariant manifolds for stochastic processes. Possible extensions of Theorem \ref{cekk_theo:SGD} are discussed in \cite[Section~2.6]{ChemnitzEngel24}. 

In summary, we have seen in this section that (stochastic) gradient descent can be analyzed using (random) dynamical systems. Selecting an adequate concept of stability turns out to be crucial. In particular, in the overparameterized setting is actually highly nontrivial to employ the concept of the leading Lyapunov exponent in Theorem~\ref{cekk_theo:SGD} for concretely given SGD more globally, i.e., beyond local stability. From a dynamics perspective this is not surprising as a similar challenge occurs, e.g., in the context of estimating Lyapunov exponents in chaotic systems. The setup we consider is a vast simplification of real machine learning applications which use advanced optimization algorithms and learning rate schedules. Still, many of the concepts introduced in this section apply to more complex settings as well. Possible generalizations of our results are discussed in detail in \cite[Section 2.7]{ChemnitzEngel24}.

\section{Mean-Field Limits, Generative Models and Optimization}
\label{cekk_sec:multi}

In Sections \ref{cekk_sec:neural} and \ref{cekk_sec:opti}, we have already seen two key examples, how problems in machine learning are deeply underpinned by dynamical systems theory. Hence, it is natural to ask how general the dynamics approach is. Can we typically convert a machine learning question so that it can be attacked via techniques from nonlinear dynamics? Here we shall show that the answer to this question is overwhelmingly positive since, almost always, a dynamics perspective is possible and helpful. In this section, we shall illustrate the dynamics viewpoint for a few exemplary topics (in no particular importance ordering) to demonstrate the generality of the approach.

\subsection{Boltzmann Machines \& Hopfield Networks as Special Cases}
\label{cekk_sec:Boltzmann}

Boltzmann machines\index{Boltzmann Machine} are effectively network variants of the classical Ising~\cite{Ising}\index{Ising Model} and Sherrington-Kirkpatrick~\cite{SherringtonKirkpatrick}\index{Sherrington-Kirkpatrick Model} models . In turn, all these systems are special cases of Hamiltonian particle systems having the Markov random field property~\cite{Isham}\index{Markov Random Field Property}. These systems are closely linked to general contact process theory~\cite{Liggett}\index{Dynamic Contact Process Theory}. Overall, all these previous classes are again just special cases of interacting particle systems (IPS)\index{Interacting Particle System (IPS)} on networks~\cite{Newman,PorterGleeson}. If learning is taken into account, all these systems become special cases of adaptive (or co-evolutionary networks)~\cite{Berneretal}\index{Adaptive Network}. The field of large IPS can at least be traced back to the studies of gas dynamics by Boltzmann~\cite{Boltzmann,Cercignani,Golse}. 

More concretely, consider a simple graph $\cG=(\cV,\cE)$ with vertex set $\cV$, edge set $\cE$ and fix the cardinality $|\cV|=M$. Denote the adjacency matrix of $\cE$ by $A=(a_{ij})_{i,j=1}^M\in\R^{M\times M}$ and let $b\in\R^M$ be a vector of parameters. Consider a phase space $\cX=\{0,1\}^M$, time domain $\cT=\N$, so that each vertex satisfies $v_i=v_i(k)\in\{0,1\}$ with $k\in\N$ and $i\in[M]=\{0,1,2\ldots,M\}$. In particular, we think of each vertex (or particle, node, or agent) as having dynamics that can either take value $0$ or $1$. Define a stochastic
dynamical system\index{Dynamical System!Stochastic Dynamical System} on $\cG$ by
\be
\label{cekk_eq:stochIPS}
v_i(k+1)=\left\{
\begin{array}{ll}
	1 & \text{with probability $p_{ik}\coloneqq f(v_i(k);A,b)$,}\\
	0 & \text{with probability $1-p_{ik}$,}\\
\end{array}
\right.
\ee 
where one has to select a, usually nonlinear, function $f$. The stochastic dynamical system~\eqref{cekk_eq:stochIPS} is a variant of classical IPS of Ising-Sherrington-Kirkpatrick type. One common choice for $f$ referred to as Boltzmann machine~\cite{AckleyHintonSejnowski} is  
\be
\label{cekk_eq:energyBM}
f(v_i(k))=\frac{1}{1+\exp(b_i+\sum_{j=1}^M a_{ij}v_j(k))}.
\ee
The link to classical statistical mechanics is observed by noticing that~\eqref{cekk_eq:energyBM} relates to an energy (or Hamiltonian) 
\be
\label{cekk_eq:Hamiltonian}
\mathcal{H}(v)\coloneqq\sum_{i=1}^M v_i b_i -\frac12 \sum_{i,j=1}^M v_i v_j a_{ij}.
\ee
In fact, the slightly earlier-developed Hopfield model~\cite{Hopfield}\index{Hopfield Model} has the same energy~\eqref{cekk_eq:Hamiltonian} except that it uses the deterministic limiting update rule
\be
\label{cekk_eq:stochIPS1}
v_i(k+1)=\left\{
\begin{array}{ll}
	1 & \text{if $p_{ik}\coloneqq f(v_i(k);A,b)>0$,}\\
	0 & \text{else.}\\
\end{array}
\right.
\ee 
A wide variety of dynamical phenomena occur in such Ising-type models, including convergence to a stationary distribution (the Boltzmann distribution) or, upon parameter variation, the existence of phase transitions~\cite{Goldenfeld}, which can usually be associated with bifurcations~\cite{GH,Kuznetsov} in a mean-field or continuum limit. The additional layer of complexity added within the machine learning context is to train the weights $a_{ij}$ and/or the elements of $b$ via given data and optimization algorithms. This explains the tendency to rename classical parameters in a dynamical system, such as $A$ and $b$, to ``hyperparameters''\index{Hyperparameters}. This convention is unfortunately misleading. It would have been more logical, much cleaner, and more consistent if one would call hyperparameters also phase/state space variables, as they also evolve dynamically during learning; see Section~\ref{cekk_sec:opti}. 

Similarly to Section~\ref{cekk_sec:neural}, recall that one splits the vertices into an input or visible set\index{Visible Vertex} $\cV_{\textnormal{in}}$ and hidden set $\cV_{\textnormal{hid}}$\index{Hidden Vertex}. Let us denote the indices in $\cV_{\textnormal{in}}$ by $\{1,2,\ldots,d\}$ with $d\leq M$. Training data consists of a family of binary vectors $\tilde{v}\in\{0,1\}^M$. Denote the probability distribution over the training set as $P^+(\cV_{\textnormal{in}})$. Suppose the Boltzmann machine reaches its equilibrium distribution $P^*(\cV)$, and denote is marginalization over $\cV_{\textnormal{hid}}$ as $P^-(\cV_{\textnormal{in}})$. Learning should make these two distributions close. Of course, in general, it is favorable to abstract and simplify this situation, i.e., we have two probability distributions $P^+$ and $P^-$, say with associated probability measures $\mu^+$ and $\mu^-$, and we want to make their distance close by learning. One option would be to pick a metric ${\bf d}$ on the space of probability measures and consider the problem
\be
\label{cekk_eq:minabstract}
\min_{A,b} {\bf d}(\mu^+,\mu^-).
\ee  
The most classical case does not quite consider a metric but uses the Kullback-Leibler divergence\index{Kullback-Leibler Divergence} or relative entropy\index{Relative Entropy}, e.g., in the case of the Boltzmann machine or Hopfield model, one often finds the minimization problem 
\be
\label{cekk_eq:minabstract1}
\min_{A,b} \sum_{V}^M P^+(V) \ln \left(\frac{P^+(V)}{P^-(V)}\right),
\ee
where we sum over all possible configurations $V\in\{0,1\}^M$. Clearly, \eqref{cekk_eq:minabstract}-\eqref{cekk_eq:minabstract1} are then just standard optimization problems that could be attacked using gradient descent or stochastic gradient descent as discussed in Section~\ref{cekk_sec:opti}. This again leads to a (stochastic/random) dynamical system, with phase or state space variables $A,b$. We highlight that, instead of considering discrete time $\cT=\N$ and a discrete state space $\cX$, there are always natural continuous-time and continuous-space analogs (and vice versa). For example, consider $x_i=x_i(t)\in\R$ and define the classical continuous-time continuous-state-space Hopfield model as 
\be
\label{cekk_eq:Hopfieldcts}
x_i'\coloneqq\frac{\txtd x_i}{\txtd t}=-\alpha x_i+\sum_{j=1}^Ma_{ij}h(x_j) + b_i,\quad i\in[M],
\ee
where $\alpha>0$ is a relaxation parameter for each individual neuron. Evidently, the system~\eqref{cekk_eq:Hopfieldcts} is just a set of differential equations on a network. The general class of network models of neural interactions has been developed in many other contexts and before Hopfield considered it in machine learning, e.g., we refer to the groundbreaking works of Amari and Wilson/Cowan ~\cite{Amari2,Amari3,WilsonCowan1,WilsonCowan} in biophysics. 

To illustrate the scientific universality of network dynamics even further, consider again a general network/graph $\mathcal{G}=(\mathcal{V},\mathcal{E})$. Assume each vertex has a $d$-dimensional smooth manifold $\mathcal{X}$ as its phase space, i.e., $x_j=x_j(t)\in\mathcal{X}$ with $t\in\mathbb{R}$ and $\dim(\cX)=d$. Common cases are $\mathcal{X}=\mathbb{R}^d$, the sphere $\mathcal{X}=\mathbb{S}^d$, or the torus $\mathbb{T}^d$. Assume that each individual vertex/node $i$ has dynamics
\begin{equation}
	\label{cekk_eq:ODE}
	x_i'\coloneqq\frac{\textnormal{d} x_i}{\textnormal{d} t} =f_i(x_i),\quad x_i(0)=x_{i,0},~i\in[M], 
\end{equation}
for sufficiently smooth sections $f_i$ mapping into the tangent bundle $\textnormal{T}\mathcal{X}$. We assume that~\eqref{cekk_eq:ODE} yields a smooth flow $\phi_t(x_{i,0})=x_i(t)$ defined for all $t\in \mathbb{R}$. As our most abstract class of continuous-time continuous-space network dynamical systems (NDS)\index{Dynamical System!Network Dynamical System}, we are going to consider the ODEs
\begin{equation}
	\label{cekk_eq:key}
	x_i' = f_i(x_i) + \sum_{j=1}^M a_{ij}~g(x_i,x_j),\qquad i\in[M],~x_i=x_i(t)\in\mathcal{X}, 
\end{equation} 
for a coupling map $g:\mathcal{X}\times \mathcal{X}\rightarrow \textnormal{T} \mathcal{X}$\index{Coupling Function}. The NDS~\eqref{cekk_eq:key} is very common in complex systems and includes\index{Boltzmann Machine}\index{Kuramoto Model}\index{Desai-Zwanzig Model}\index{van-der-Pol Oscillator}\index{FitzHugh-Nagumo Model}\index{Hopfield Model}
	\begin{flalign}
		&\textnormal{(Boltzmann)} 
		&& \mathcal{X}=\mathbb{R}^6, && f_i(x_i)=(x_{i,1},x_{i,2},x_{i,3},0,0,0)^\top, \nonumber \\
		&&&a_{ij}\equiv \frac{1}{M},&&~g_{4,5,6}(x_i,x_j)=\gamma\frac{(x_i-x_j)}{\|x_i-x_j\|^3} \hspace{10mm} \\
		&\textnormal{(Kuramoto)} 
		&& \mathcal{X}=\mathbb{S}^1\coloneqq\mathbb{R}/(2\pi \mathbb{Z}),&& f_i(x_i)\equiv 
		\omega_i\in\mathbb{R}, \nonumber \\
		&&& a_{ij}\equiv \frac{K}{M}, &&g(x_i,x_j)=\sin(x_j-x_i) \label{cekk_eq:Kuraoriginal} \\
		&\textnormal{(Desai-Zwanzig)} &&\mathcal{X}=\mathbb{R}, &&f_i(x_i)=-V'(x_i), \nonumber\\ 
		&&& a_{ij}\equiv \frac{K}{M},
		&&g(x_i,x_j)=x_j-x_i,\\
		&\textnormal{(coupled vdP/FHN)}&&\mathcal{X}=\mathbb{R}^2, && f_i(x_i)=(x_{i,2}-\frac13 x_{i,1}^3+x_{i,1},
		-\varepsilon x_{i,1})^\top, \nonumber\\ 
		&&& a^M_{ij}\equiv \frac{K}{M}, && g_1(x_i,x_j)=x_j-x_i,\\
		&\textnormal{(Hopfield)} &&\mathcal{X}=\mathbb{R}, &&f_i(x_i)\equiv -\alpha x_i-b_i, \nonumber \\
		&&& a_{ij}\equiv \frac{K}{M}, &&g(x_i,x_j)=\frac{1}{1+\textnormal{e}^{-x_j}}
	\end{flalign}
	\begin{flalign}
		&\textnormal{(Hegselmann-Krause)} &&\mathcal{X}=[0,1], &&f_i(x_i)\equiv 0, \nonumber \\
		&&& a_{ij}\equiv \frac{K}{M}, && g(x_i,x_j)=(x_j-x_i){\bf 1}_{\{-c_i\leq x_j-x_i\leq d_i\}}\\ 
		&\textnormal{(Cucker-Smale)}&&\mathcal{X}=\mathbb{R}^2, && f_i(x_i)=(x_{i,2},0)^\top, \nonumber \\
		&&& a_{ij}\equiv \frac{K}{M}, && g_2(x_i,x_j)=\frac{x_{j,2}-x_{i,2}}{1+\|x_{i,1}-x_{j,1}\|^\alpha},
\end{flalign}
where $i\neq j$, $a_{ii}^M=0$, the sign of $\gamma\in\mathbb{R}$ determines whether we consider Coulomb or gravitational interaction in the collision-less Boltzmann gas dynamics case~\cite{Neunzert}, $K>0$ is a parameter controlling the coupling strength, the fixed frequencies $\omega_i$ for the Kuramoto oscillators are usually sampled from a given probability distribution~\cite{Kuramoto,Strogatz1}, $V$ is a given confining potential for particles in Desai-Zwanzig~\cite{DesaiZwanzig}, $\varepsilon>0$ is a small parameter for the time scale separation in the diffusively coupled van-der-Pol/FitzHugh-Nagumo~\cite{FitzHugh,Nagumo,vanderPol} oscillator case~\cite{Winfree1,SomersKopell}, $\alpha>0$ and $b_i\in\mathbb{R}$ are parameter pairs controlling decay to the rest state and external inputs for the individual neurons in the Hopfield~\cite{Hopfield} case, $c_i,d_i>0$ are generosity and competition thresholds for the indicator function of each agent in Hegselmann-Krause~\cite{HegselmannKrause} opinion dynamics, and finally $\alpha>0$ is a parameter controlling the nonlinearity for Cucker-Smale~\cite{CuckerSmale1} model for flocking/swarming. In fact, this list of models in the form~\eqref{cekk_eq:key} is highly incomplete, there are many more examples, e.g., in ecology (e.g.~food webs), in neuroscience (e.g.~spiking neural networks), and effectively in all other sciences. It is quite frequent that separate scientific communities have developed their own independent names and terminology to study the class of NDS~\eqref{cekk_eq:key}. In fact, this makes rediscoveries quite likely, which could be avoided by just recognizing that the abstract mathematical structure is the same. \index{Hegselmann-Krause Model} \index{Cucker-Smale Model}

\subsection{Mean-Field Limits of Interacting Particle Systems}
\label{cekk_sec:mean_field}\index{Mean-Field Limit}

From the perspective of dynamical systems, it is much more practical and effective to treat the class of NDS~\eqref{cekk_eq:key} as one common field of study. There is a multitude of results available already, although it is generally difficult to rigorously prove statements for finite fixed but large $M$ beyond special cases such as Hamiltonian/integrable structure, symmetries, or multiscale reduction. One device to sometimes ameliorate this difficulty is to pass to a mean-field or continuum limit $M\ra \I$, which we shall discuss next. Here we shall not discuss the continuum limit, which arranges the vertices on a discrete structure and then considers a limit. 

\subsubsection*{All-to-all Coupled Systems}\index{All-To-All Coupling}

For didactic simplicity, consider first the case for all-to-all coupled systems~\cite{Golse} 
\benn
a^M_{ij}=1~\text{ if $i\neq j$}\qquad \text{and}\qquad a^M_{ii}=0~~\text{ if $i\in[M]$},
\eenn
where we have introduced the index $M$ to indicate that the adjacency matrix depends on the number of particles/neurons. Consider~\eqref{cekk_eq:key} for simplicity with $f_i\equiv0$ and $a^M_{ij}\equiv1/M$ for $i\neq j$, $\mathcal{X}=\mathbb{R}^d$, so that
\begin{equation}
	\label{cekk_eq:kin}
	x_i' = \frac1M\sum_{j=1}^M g(x_i,x_j), 
\end{equation} 
where $a^M_{ii}=0$ is incorporated by assuming $g(x,\tilde{x})=-g(\tilde{x},x)$. Suppose we had
\begin{equation*}
	\frac1M\sum_{j=1}^M g(x_i(t),x_j(t)) \rightarrow \int_{\mathbb{R}^d} g(x_i(t),\tilde{x})~\mu(t,\textnormal{d} \tilde{x})\,,\qquad \text{for }M\rightarrow \infty,
\end{equation*}
for a measure $\mu(t,x)$, or even better for a measure with a density $\mu(t,x)=u(t,x)~\textnormal{d} x$.
Then one can replace~\eqref{cekk_eq:kin} by the characteristic ODE
\begin{equation}
	\label{cekk_eq:kin1}
	x' = \int_{\mathbb{R}^d} g(x,\tilde{x})~\mu(t,\textnormal{d} \tilde{x}). 
\end{equation}
The equation \eqref{cekk_eq:kin1} is the characteristic ODE for the first-order PDE
\begin{equation}
	\label{cekk_eq:Vlasov}
	\partial_t \mu = -\nabla_x \cdot (\mu \mathcal{V}[\mu]),\qquad \mathcal{V}[\mu](t,x)\coloneqq\int_{\mathbb{R}^d} g(x,\tilde{x})~\mu(t,\textnormal{d} \tilde{x}).    
\end{equation} 
The PDE~\eqref{cekk_eq:Vlasov} simplifies further if $\mu$ has a density $u$, leading to a more classical Vlasov equation\index{Vlasov Equation}
\begin{equation}
	\label{cekk_eq:Vlasov1}
	\partial_t u = -\nabla_x \cdot (u V[u]),\qquad V[u](t,x)\coloneqq\int_{\mathbb{R}^d} g(x,\tilde{x})~u(t,\tilde{x})~\textnormal{d} \tilde{x}.    
\end{equation} 
Studying approximation properties between~\eqref{cekk_eq:Vlasov1} and~\eqref{cekk_eq:kin} requires the empirical measure
\begin{equation*}
	\delta_{M}(t)\coloneqq\frac{1}{M}\sum_{j=1}^M \delta_{x_j(t)},
\end{equation*}   
where $\delta_{x_j(t)}$ is the usual Dirac measure at $x_j=x_j(t)$. Note that $\delta_M(t)$ is a probability measure tracking an averaged state of a finite number of interacting agents/particles. It is natural to compare it to a probability density $u$ and its associated measure $\mu$. The question is whether there exists a metric ${\bf d}$ on probability measures such that
\begin{equation}
	\label{cekk_eq:approx}
	\lim_{M\rightarrow \infty}\sup_{t\in[0,T]}{\bf d}(\delta_M,\mu)=0.
\end{equation}
Classical choices for ${\bf d}$ are the Wasserstein(-Monge-Kantorovich-Rubinstein)\index{Wasserstein Distance}~\cite{Golse,Villani1} metrics
\begin{equation}
	\mathcal{W}_{p}(\nu,\mu)\coloneqq \inf_{\pi\in \Pi(\mu,\nu)} \left(\int_{\mathbb{R}^d\times \mathbb{R}^d} \|x-y\|^p~\pi(\textnormal{d} x~\textnormal{d} y)\right)^{1/p},
\end{equation}  
where $\Pi(\mu,\nu)$ is the space of all couplings between $\mu$ and $\nu$, i.e., the $\mu$ and $\nu$ are the marginals on the first and second argument of the measures in $\Pi(\mu,\nu)$. One approach to justify~\eqref{cekk_eq:approx} is to first study existence and uniqueness for the Vlasov equation \eqref{cekk_eq:Vlasov} and/or \eqref{cekk_eq:Vlasov1}. Once this is proven, one studies continuous dependence with respect to the initial condition in ${\bf d}$. If this holds, e.g., 
\begin{equation}
	\label{cekk_eq:approx1}
	\lim_{M\rightarrow \infty}\sup_{t\in[0,T]}{\bf d}(\delta_M(t),\mu(t))=0\quad \text{if}\quad \lim_{M\rightarrow \infty}{\bf d}(\delta_M(0),\mu(0))=0,
\end{equation}
then evolving the empirical measure via the system of finite-dimensional ODEs~\eqref{cekk_eq:kin} yields a good approximation to the mean-field limit Vlasov equation and vice versa; one also says that the Vlasov equation is strictly derivable in $[0,T]$~\cite{BraunHepp,Golse,KuehnBook1,Neunzert}. 

A typical result in the 1-Wasserstein (or bounded Lipschitz) metric for all-to-all coupled systems with a Lipschitz coupling function $g$ regarding continuous dependence of the form~\eqref{cekk_eq:kin} are Dobrushin-type estimates\index{Dobrushin-Type Estimate}~\cite{Dobrushin,Dobrushin1} for two solutions $\mu_{1,2}(t)$ given by 
\begin{equation}
	\label{cekk_eq:Dobrushin}
	\mathcal{W}_{1}(\mu_1(t),\mu_2(t)) \leq \textnormal{e}^{2L|t|} \mathcal{W}_{1}(\mu_1(0),\mu_2(0))\,,
\end{equation}
where $L>0$ is the Lipschitz constant of $g$. From~\eqref{cekk_eq:Dobrushin}, one concludes that the Vlasov equation is strictly derivable on $[0,T]$. 

We also remark that the most common limit relevant in ML are deep neural networks of general heterogeneous networks that we shall study in the next section. Yet, there are classical cases, where all-to-all coupling already appears. For example, the training of a feed-forward neural network (see Section~\ref{cekk_sec:feedforward}) with a single layer with $M$ neurons can be interpreted as an IPS with all-to-all coupling \cite{Chizat18, Sirignano20, Gess25}. Here, every neuron is considered to be an individual particle, whose state is given by the weights of the ingoing and outgoing edges and its bias. The interactions arise from the gradient descent update rule \eqref{cekk_eq:GD}. Thus, the limit $M\to \infty$ describes the training of a one-layer neural network\index{One-Layer Neural Network} with infinite width, the opposite limit of the infinite-depth neural networks studied in Section~\ref{cekk_sec:neuralODEs}.

\subsubsection*{Singularly Coupled Systems, Graphops and Digraph Measures}

Building upon more abstract work in kinetic theory~\cite{Neunzert,Golse,Jabin}, one can actually specialize to particular models to obtain concrete and applicable results. As an illustration consider the Kuramoto\index{Kuramoto Model} case~\cite{Lancellotti}, where one obtains for the Vlasov equation
\begin{equation*}
	\label{cekk_eq:KurVlas}
	\partial_t u = - \partial_x(uV[u]),\; V[u](x,t,\omega)=\omega + K \int_0^{2\pi}\int_\mathbb{R} \sin(\tilde{x}-x) u(\tilde{x},t;\tilde{\omega})\zeta(\tilde{\omega})~\textnormal{d} \tilde{\omega} ~\textnormal{d} \tilde{x},
\end{equation*}
where $\zeta$ is a given fixed probability distribution of the intrinsic frequencies $\omega_i$ of the oscillators, i.e., the distribution specifying the intrinsic dynamics $f_i(x_i)=\omega_i$. For Kuramoto (and suitable variants), it has been shown recently that one can take into account graphs $\mathcal{G}^M$ modelled via a graphon $G$~\cite{Lovasz}, i.e., the adjacency matrix is derived from a kernel 
\begin{equation}
	\label{cekk_eq:graphon}
	a^M_{ij}=\int_{\mathcal{I}_i^M\times \mathcal{I}_j^M} G(x,y)~\textnormal{d} x~\textnormal{d} y,\qquad G:\mathcal{I}\times \mathcal{I}\rightarrow \mathbb{R},
\end{equation} 
where $\{\mathcal{I}_i^M\}_{i=1}^M$ is a subdivision of the unit interval $\mathcal{I}=[0,1]$. Under technical assumptions on $G$ and $\zeta$, one gets a modified Vlasov equation~\cite{AyiDuteil,ChibaMedvedev,KaliuzhnyiVerbovetskyiMedvedev1} for $u=u(x,t,\omega,z)$
\begin{equation*}
	\label{cekk_eq:KurVlas1}
	\begin{split}
		\partial_t u &= - \partial_x(uV[u]), \\
		V[u](x,t,\omega,z) &=\omega + K \int_\mathcal{I}\int_0^{2\pi}\int_\mathbb{R}  \sin(\tilde{x}-x) u(\tilde{x},t,\tilde{\omega},\tilde{z})\zeta(\tilde{\omega})G(z,\tilde{z})~\textnormal{d} \tilde{\omega}~\textnormal{d} \tilde{x}~\textnormal{d} \tilde{z},
	\end{split}
\end{equation*}
where $z\in \mathcal{I}$ tracks the heterogeneity of the graph/network. Instead of mean-field Vlasov equations, a similar approach can also be exploited for continuum limit PDEs including graphons\index{Graphon}~\cite{Medvedev3,Medvedev2,KaliuzhnyiVerbovetskyiMedvedev,KuehnThrom2}. Yet, notice that~\eqref{cekk_eq:graphon} makes the strong assumption that the graph $G$ has as its limit a well-defined graphon $G\in L^p(\mathcal{I}\times \mathcal{I},\mathcal{I})$ for $p\in[1,\infty)$, which is not the case for large classes of graph limits~\cite{BackhauszSzegedy}. Yet, even in more general cases, there are extensions possible to graph operators (graphops) as discussed in~\cite{KuehnGraphops,GkogkasKuehn,KuehnXu,GkogkasKuehnXu,GkogkasKuehnXu1}\index{Graphop}. Considering additional stochastic forcing terms for the ODEs giving SDEs leads to (Vlasov-)Fokker-Planck-Kolmorogov (FPK) equations~\cite{Risken,Frank,Pavliotis1,Sznitman}\index{Fokker-Planck-Kolmogorov (FPK) Equation}. 
In the course of the research project, which led to this book chapter, we have proven one particular variant of mean-field limits containing a generalization of graph limits via so-called digraph measures (DGM)~\cite{KuehnXu}.\index{Digraph Measure (DGM)} 

Here, we briefly sketch this result for illustration purposes from~\cite{KuehnPulido}. We consider a stochastic variant of the NDS \eqref{cekk_eq:key}\index{Dynamical System!Network Dynamical System} given by
\be
\label{cekk_eq:SDEIPS}
\txtd x_i = f(x_i) \txtd t + \frac{1}{M} \sum_{j=1}^{M} a^{M}_{ij} g(x_i, x_j) \txtd t + \frac{1}{M} \sum_{j=1}^{M} \hat{a}^{M}_{ij} h(x_i, x_j) \txtd W^i
\ee		
with initial conditions $x_i(0)=x_i^0$ for $i\in [M]$, $(a_{i,j}^M)_{i,j=1}^M$ is the usual adjacency matrix of the system representing the basic deterministic network interaction between vertices $i$ and $j$, $\hat{a}^{M}_{ij}$ is an adjacency matrix representing how the noise is influenced by the interactions between $i$ and $j$, and finally $W^i=W^i(t)$ for $i\in[M]$ are independent identically distributed Brownian motions. We shall always assume again that $f,g,h$ are Lipschitz. To capture the heterogeneity of the graphs, we again introduce an interval $\cI=[0,1]$, which is going to form a location of each type of vertex via a variable $u\in \cI$. Since we are working with DGMs, and in our problem we have a graph represented by a matrix, we need to approximate this graph by a DGM. Consider a partition of the interval $\cI=[0,1]$ as given by $I_i^M=]\frac{i-1}{M},\frac{i}{M}]$, for $1<i\leq M$, and $I_1^M=[0,\frac{1}{M}]$. Then we introduce digraph measures associated to $a^M_{i,j}$ and $\hat{a}^M_{i,j}$ given by
\be
\begin{split}
	\eta^u_{A^M}(v)\coloneqq\sum_{i=1}^M\textbf{1}_{I_i^M}(u)\sum_{j=1}^M\frac{a^M_{i,j}}{M}\delta_{\frac{j}{M}}(v),\\
	\eta^u_{\hat{A}^M}(v)\coloneqq\sum_{i=1}^M\textbf{1}_{I_i^M}(u)\sum_{j=1}^M\frac{\hat{a}^M_{i,j}}{M}\delta_{\frac{j}{M}}(v),
\end{split}
\label{cekk_eq:dgm}
\ee
where $\textbf{1}_{\frac{j}{M}}(x)$ denotes the indicator function and ${\frac{j}{M}}$ serves as the representative of the set $I_j^M\coloneqq(\frac{i-1}{M},\frac{i}{M}]$ subdividing $\cI$. We shall assume that both graph sequences for the interactions have limits in the following sense: graphs $\{A^M\}_{M\geq1}$ converge to a limiting measure $\eta$ if ${\bf d}_\infty(\eta_{A^M},\eta)\rightarrow 0$ as $M\rightarrow\infty$, where ${\bf d}_\infty$ is the bounded Lipschitz distance on measures. Similarly, $\hat{A}^M$ shall converge to a digraph measure $\hat{\eta}$. Instead of the ODE characteristic equation from the deterministic IPS case, one considers the following independent processes:
\be
\begin{split}
	X_u(t) =& X_u(0) + \int_{0}^{t} f(X_u(s)) ~\txtd  s + \int_{0}^t \int_\cI \int_{\cX} g(X_u(s), y) ~\mu_{v, s}(\txtd y)\eta^u(\txtd v) \txtd s \\
	&+ \int_{0}^t \int_\cI \int_{\mathbb{X}} h(X_u(s), y)~\mu_{v, s}(\txtd y)\hat{\eta}^u(\txtd v) \txtd W_s^u,
\end{split}
\label{cekk_eq:indep}
\end{equation}
where we define $\mu_{u,t} = \mathcal{L}(X_u(t))$ as the law of $X_u(t)$ for $u\in \cI$. The initial conditions $X_u(0)=X_u^0$ are independently and identically distributed under a probability measure $\bar{\mu}^0\in\mathcal{P}(\cX)$. The two measures $\eta,\hat{\eta}$ have a disintegration into a family of measures $\left\{\eta^u\right\}$ with $u \in \cI$, $\left\{\hat{\eta}^u\right\}$ with $u \in \cI$. These are called fiber measures \cite{BackhauszSzegedy}\index{Fibre Measure} in direct analogy to the theory of graphops~\cite{KuehnGraphops}\index{Graphop}. Effectively, these measures represent the adjacency matrices $a^M_{i,j}$ and $\hat{a}^M_{i,j}$, and each fiber measure describes a local connectivity of a given vertex with label $u\in\cI$. The lowest-order mean-field approximation is obtained if we average over the heterogeneity of the graph, i.e., we take the average of the heterogeneous empirical measures
\benn
\bar{\mu} \coloneqq \int_\cI \mu_{u,t} ~ \txtd u.
\eenn 
The main theorem of~\cite{KuehnPulido} states that we then again converge (in a rather technical metric) to a nonlocal Vlasov-Fokker-Planck\index{} equation 
\be
\begin{split}
&\partial_t \bar{\mu}_{t} + \partial_x \left(\bar{\mu}_{t} f(x) + \int_I{\mu}_{u,t} \int_I \int_{\mathbb{X}} g(x,y) ~\mu_{v,t}(\text{d}y) \eta^u(\text{d}v)\text{d}u\right) \\
&+ \frac{1}{2}\partial_x^2\left(\int_I {\mu}_{u,t} \left[\int_I \int_{\mathbb{X}} h(x,y) ~\mu_{v,t}(\text{d}y) \hat{\eta}^u(\text{d}v)\right]^2~\text{d}u\right)=0.
\end{split}
\ee
Of course, one can push these types of results even further beyond graph heterogeneity and leading-order mean-field limits. For example, one could relatively directly combine this theory with stochastic corrections to the mean-field mesoscopic PDE level. Then one is going to obtain stochastic PDEs. Furthermore, if one scales the interaction radius between particles to moderate instead of long-range interactions, one can obtain reaction-diffusion or fluid-type PDEs on the macroscopic level. These classes also arise in the limit $M\ra \I$ if one considers macroscopic observables and derives their differential equations. 

In summary, one should rather think of the classical machine learning network models as very special cases of the field of NDS, dynamical systems, and in the limit of large networks as particular cases of important partial-integro-differential equations. This makes all classical tools of these fields immediately applicable to ML/AI. We also refer to the chapter in this book for the SPP2298 project ``Assessment of Deep Learning through Meanfield Theory'' by PI Herty and collaborators, where the same general strategy of utilizing mean-field limits of IPS in ML problems is discussed. 

\subsection{Recurrent Neural Networks, Transformers, and Beyond}
\label{cekk_sec:RNNs}
Now one could argue that there is an ongoing recent stream of network architecture and algorithm refinements in ML/AI, and that this could require to rethink the NDS model class \eqref{cekk_eq:key} discussed above.
A closer mathematical look indicates that a very mild generalization already suffices. 

For example, consider recurrent neural networks (RNNs)\index{Recurrent Neural Network (RNN)} and their various incarnations. One typical formulation of the dynamical update rule for an RNN is given by
\be
\label{cekk_eq:delayRNN}
x_{k+1}= Ax_k + B \tilde{G}(x_{k-1})+CX + b_k  
\ee
where $x_k\in\R^n$, $A,B,C$ are matrices, $b_k$ are biases, $\tilde{G}:\R^n\rightarrow\R^n$ is the given nonlinearity (or activation function), and $X$ represents the input data. Equation~\eqref{cekk_eq:delayRNN} arises as the time discretization of a delay differential equation with a suitable fixed delay. Clearly~\eqref{cekk_eq:delayRNN} defines a dynamical system with memory, which is actually an old idea~\cite{Roberts3,Anderson2} predating significantly the more modern use of time delay neural networks~\cite{Waibel}; see also Section~\ref{cekk_sec:delay} for the importance of delays for embedding/approximation properties. 

One of the main variants associated to~\eqref{cekk_eq:delayRNN} are long short-term memory (LSTM) architectures~\cite{HochreiterSchmidhuber}\index{Long Short-Term Memory (LSTM) Architecture}. Although it is technically tedious, also LSTMs can be associated with dynamical systems with delay~\cite{Sherstinsky}. Effectively, any recurrent network can also be viewed from a network dynamics perspective as an NDS with longer-range connections. In fact, we have already seen Hopfield and Boltzmann-machine examples above, which are all-to-all connected in generic cases. Other examples with a direct link are echo-state networks (ESN)\index{Echo-State Network} and reservoir computers\index{Reservoir Computing}~\cite{Jaeger,OzturkXuPrincipe,Appeltantetal}. 

Instead of going into these classical architecture ideas in more detail, one could try to object that the conversion to a dynamics viewpoint is not applicable for ``more recent'' ML/AI architectures and algorithms. Usually, this objection only arises due to a real-life practical time delay in making a system mathematically more precise. To illustrate this, let us consider ``transformers''\index{Transformer} (or attention networks), which are often attributed to~\cite{Vaswanietal}. For illustration purposes, we shall use continuous time (``deep transformers''), but very similar remarks apply to discrete time. Let $x_i=x_i(t)\in\R^d$ and define the transformer via the flow of the ODEs
\be
\label{cekk_eq:transformer}
x_i'=\sum_{j=1}^M \frac{\exp(\cM_1(t)x_i\cdot \cM_2(t)x_j)}{\sum_{l=1}^M \exp(\cM_1(t)x_i\cdot \cM_2(t)x_j)}\cM_3(t)x_j,\quad i \in[M],
\ee  
which uses the ``soft-max'' or ``self-attention'' mechanism for $\cM_{1}(t),\cM_{2}(t)\in\R^{M_0\times M}$ (for $M_0\leq M$) and $\cM_{3}(t)\in\R^{M\times M}$ being given matrices. Of course, many variants of the nonlinearity are possible. Clearly, we  recognize~\eqref{cekk_eq:transformer} as another variant of the standard NDS~\eqref{cekk_eq:key}. In fact, all transformer-type models are again just forms of IPS. We are also now used to the fact that the coefficients could be time-dependent, as this links to the structural network properties discussed in Section~\ref{cekk_sec:neural} and to training as discussed in Section~\ref{cekk_sec:opti}. 
Regarding training, we also mention an area, which we do not cover here due to space constraints: namely, one can also use the dynamics of the IPS viewpoint discussed above for node limits in the context of training. Indeed, if one views edge weights and biases as the independent dynamical variables during training, one can again take mean-field limits of IPS. This aspect is covered in many recent works, see e.g.~\cite{Fornasieretal,RotskoffVandenEijnden,Sirignano20}. The main challenges for dynamics \emph{of} networks remain similar to the dynamics \emph{on} networks covered in Section~\ref{cekk_sec:mean_field}, i.e., to justify a mean-field or continuum limit as the number of edges/parameters tends to infinity.

\subsection{Generative Models}
\label{cekk_sec:gen_models}

So far, we have typically assumed that our network trains on a certain given data set $(x_i, y_i)_{i \in \mathbb{N}}$ and then estimates, in some sense, the conditional probability $\mathbb{P}(Y | X =x)$, seeing our input/observable $X$ and target/output $Y$ as random variables. This is often called \emph{discriminative} modeling.
In contrast, in \emph{generative} models the input/observable data are generated according to an estimation of $\mathbb{P}(X | Y=y)$ in combination with $\mathbb{P}(Y)$, yielding an estimator of the joint probability distribution of observable and target, in particular $\mathbb{P}(Y | X =x)$ via Bayes' rule; see \cite{NgJordan} for a distinction of these two classes. Generative models are highly relevant for selecting features, facilitating classification and increasing model accuracy, or generating new realistic data samples to proceed with the learning process.

In the context of (deep) learning, there are several important variants of generative models (see, e.g., \cite{Harshvardhanetal} for a review), including Gaussian mixture models, hidden Markov models, latent Dirichlet allocation, Boltzmann machines, variational autoencoders, generative adversarial networks or generative diffusion models. Boltzmann machines have been discussed in Section~\ref{cekk_sec:Boltzmann}. From a dynamical perspective, the other most interesting classes are hidden Markov models (HMMs), generative adversarial networks (GANs) and generative diffusion models (GDMs).

HMMs\index{Hidden Markov Models} \cite{CappeMoulinesRyden} consider a time series of observations as a Markov chain $\{X_k\}_{k \geq 0}$ with transition kernel $Q$, i.e.,
$$ \mathbb{P}(X_{k+1} \in A | \mathcal F_k ) = Q(X_k, A)$$
for measurable sets $A$ and a corresponding filtration $\mathcal F_k$. The full HMM now adds a transition kernel $G$ from the observation variable $X$ to the target variable $Y$ to obtain a kernel for the joint Markov chain $\{X_k, Y_k\}_{k \geq 0}$
$$ T[(x,y), C] = \iint\limits_C Q(x, \text{d} x') G(x', \text{d} y'), $$
where $C$ is a measurable set from the product space.
Now, this joint Markov chain can be used to generate the aspired joint distribution of the input observables $X$ and the output target $Y$. 
The dynamical point of view becomes relevant as one would hope for an invariant, ergodic probability distribution $\pi(x,y)$ that can be approximated by sampling from the Markov chain with an appropriate kernel $T$ (which should be estimated based on some first observed time series inducing a guess of $Q$ and $G$). Now, several Monte Carlo algorithms, like accept-reject or Metropolis-Hastings schemes, can be applied to achieve this task; the theoretical underpinning of these approaches is well-studied, in terms of classical detailed balance conditions/reversible Markov chains \cite{RobertCasella} but also more recent methods including irreversibility \cite{Ottobre}.

Generative adversarial networks (GANs)\index{Generative Adversarial Networks} have been proposed in \cite{Goodfellowetal} and, since then, have seen tremendous impact in theory and applications. In this approach, one trains two models: a generative model capturing the data distribution and a discriminative model estimating
the probability that a certain data sample belongs to the set of actual training data rather than the generative model. This framework understands the training procedure as a two-player game where the generative model tries to win against the real training data. This approach does not deploy an underlying Markov chain but can rather be associated with iterated games, immediately linked to dynamical systems theory. 

A starting point for investigating the corresponding regret-minimizing learning dynamics in zero-sum games has been the model of replicator dynamics, arguably the most well-studied structure in evolutionary game theory~\cite{Weibull}. The replicator structure is the continuous-time analog of the Multiplicative Weights Update (MWU), a meta-algorithm which is well-known for its regret guarantees~\cite{Aroraetal}. 
In the specific case of two player zero-sum games, the replicator equations are given by the ODE
\begin{equation}  \label{eq:replicator2}
\begin{array}{rl}
	\dot x_i &= x_i \left(\{Ay\}_i - x^{\top} Ay\right)\,,\\
	\dot y_j &= y_j \left(\{Bx\}_j - y^{\top} Bx\right)\,,
\end{array}
\end{equation}
where $(x, y) \in \Delta_n \times \Delta_m$, denoting the $n$- and $m$-dimensional simplices respectively, and $A$ and $B$ are pay-off matrices. A (mixed) Nash equilibrium, the key object in game dynamics, is precisely an (interior) equilibrium of the ODE~\eqref{eq:replicator2}, which is typically elliptic, i.e., only reached in average (see, e.g., \cite{Hofbauer}). 
The impact of stochastic noise on such models, as clearly relevant also for the algorithmic situation of GANs, has been studied in \cite{EngelPiliouras}, which was also a publication in the course of our SPP2298 research project. The result shows that the choice of noise is crucial for observing an interior Nash equilibrium or not, i.e., for the algorithm to converge, at least in average, to a favorable solution.

Similarly (also to the previous subsections), the discrete-time problem of generative models aiming to learn a probability distribution can be related to continuous-time models that are easier to study, in this case, SDEs of the form
\begin{equation}
\label{cekk_eq:SDE}
\text{d} X_t = b(X_t, t) \text{d} t + \epsilon \sigma(X_t, t) \text{d} W_t,
\end{equation} 
where $b$ is a drift vector field and $W_t$ is a vector of independent Brownian motions with a scaling parameter $\epsilon$ controlling the noise level. Such models may be understood as generative diffusion models (GDMs)\index{Generative Diffusion Models} \cite{Caoetal}. The existence of ergodic, invariant probability distributions, giving the typical behavior of the generated time series, can now be studied around finding (potentially unique) solutions of the stationary Fokker-Planck equation
$$ \mathcal L^* p = 0,$$
where $\mathcal{L}^*$ is the forward Kolmogorov operator associated to the SDE~\eqref{cekk_eq:SDE}.
Here, one can refer to an extremely well-studied problem in stochastic dynamics in order to understand and control the outcomes of generative procedures. 
A recent study \cite{HessMorris} relates such generative diffusion models to associative memory models via perturbations of Morse-Smale flows, taking $\epsilon$ in \eqref{cekk_eq:SDE} small.
In this context, the extensively studied problem of metastability \cite{BovierdenHollander} enters the realm of generative modeling, allowing for many new insights via another well-built theory from (stochastic) dynamical systems.

A specifically interesting variant of GDMs are score-based generative models (SGMs) which transfer the distribution $p_0(x)$ of data $X(0)$ along the dynamics of an SDE~\eqref{cekk_eq:SDE} to a distribution $p_T(x)$ that can be sampled from efficiently. The corresponding reverse SDE, which includes a score term $\nabla_x \log p_t(x)$ and induces the solution process $Y_t = X_{T-t}$, can then be used to generate new samples (see, e.g., \cite{Songetal}).
Recently, the robustness of such models has been studied by understanding the Markov solution process of the (reverse) SDE as a random dynamical system (RDS), i.e., via the same formalism that we have introduced in Section \ref{cekk_subs:RDS} to study the stability of stochastic gradient descent. In \cite{ChandramoorthyDeClercq}, the authors show that, if the leading Oseledets spaces associated with the largest Lyapunov exponents of the RDS solutions align with the boundary of the manifold of data, then the generating scheme of the model is robust under small errors. Here, robustness is understood in terms of the support of the generated distribution: for sufficiently small perturbations, the predicted target is only supported where the actual target density is supported. The dynamical conditions for the proposed theory crucially include attraction to the support by the vector field $b$ in the SDE~\eqref{cekk_eq:SDE} and general regularity assumptions to apply mutiplicative ergodic theory (cf.~\cite{Arnold98}). Hence, this is another example of how ergodic theory and dynamical systems are highly relevant to understand the stability of learning procedures.

\subsection{Backpropagation, Optimization and Gradients}
\label{cekk_sec:backpropagation}

We have seen that information propagation and learning concepts from machine learning can essentially always be interpreted as special cases of dynamical phenomena. Yet, one may wonder whether more static-looking concepts in the area can also be translated. To illustrate this, let us just take one example and look at the very classical idea of backpropagation\index{Backpropagation}. Probably the algorithmically cleanest description of backpropagation has been given within the field of automatic differentiation~\cite{Griewank1}. Suppose we have a simple composition of differentiable scalar functions 
\benn
F(x_0)=f_1(f_0(x_0)),\qquad x_0\in \R,~x_1\coloneqq f_0(x_0)\in\R,~x_2:=f_1(x_1)\in\R.
\eenn
Computing the derivative is done by the chain rule, so 
\benn
\frac{\partial F}{\partial x_0}=\frac{\partial f_1(x_1)}{\partial x_1} \frac{\partial f_0(x_0)}{\partial x_0}=\frac{\partial x_2}{\partial x_1} \frac{\partial x_1}{\partial x_0}.
\eenn
There is a choice how to evaluate the expression. We could evaluate the chain rule ``inside-out'' starting with the innermost part $\frac{\partial x_1}{\partial x_0}$ and then compute the next outer part $\frac{\partial x_2}{\partial x_1}$. This is called forward accumulation in automatic differentiation. Alternatively, one could proceed via reverse accumulation and first evaluate $\frac{\partial x_2}{\partial x_1}$ and then $\frac{\partial x_1}{\partial x_0}$, evaluating the chain rule ``outside-in''. It is well-known that reverse accumulation is the key step in backpropagation. The reason is that for multi-dimensional functions $F:\R^{N_1}\ra \R^{N_2}$ with $N_1\gg N_2$, reverse accumulation is much more efficient than forward accumulation. The case $N_1\gg N_2$ is precisely the one encountered in ML during training. At first sight, there seems to be no dynamical principle behind backpropagation, but this is not the case. It helps to think more generally and consider an arbitrary composition of functions
\benn
x_{k+1}=f_k( f_{k-1}(f_{k-2}( \cdots f_0(x_0) ))=:f(x_0,k).
\eenn   
Quite clearly, the last equation just describes the time-$k$ map of a non-autonomous discrete-time dynamical system\index{Dynamical System!Discrete Time} $x\mapsto f(x,k)$; see~\cite{Poetzsche2} for a broader background on discrete-time non-autonomous dynamics. The difference to a standard autonomous iterated map is that we must allow that the iterated map also depends upon the time step $k$. Of course, if it is independent of $k$ then we just have a more classical autonomous iterated map. Let us think of the initial condition $x_0$ more generally as one possible instance of parameters $p$ such that we get
\benn
x\mapsto f(x,k;p).
\eenn 
Crucially, we could then ask, what does it dynamically mean to compute the derivative of the iterated map with respect to the parameters $p$, i.e., $\txtD_pf$ along the iterates of the map $x_j=f_{j-1}(x_{j-1};p)$. This is well-known in dynamical systems as the variational equation with respect to $p$. It amounts to a local linearization along a solution trajectory. To evaluate it, one should compute
\benn
\txtD_p f_k( f_{k-1}(f_{k-2}( \cdots f_0(x_0;p) );p),
\eenn  
which is essentially the same as the standard backpropagation problem in machine learning if one identifies the parameters $p$ with the unknown weights and biases. The forward accumulation terminology now becomes even more intuitive as evaluating the chain-rule inside-out precisely corresponds to evaluating the variational equation forward along the trajectory, while reverse accumulation corresponds to evaluating the variational equation backward (or ``reverse'') along the trajectory. In particular, the key insight of backpropagation really is that it is equivalent and sometimes beneficial to look in reverse time for an object that could be considered in both time directions. In dynamics, this is an extremely old idea and has been used many times. Just as an illustration, the importance of the equivalence of propagating small perturbations forward and backward already features in the work of Helmholtz on the principle of least action in 1887~\cite{Helmholtz} (it is likely that similar ideas appear in various other contexts even far before that time as it is a generic dynamical principle to compare forward/backward motions and to utilize local linearizations along solutions)\footnote{It is actually interesting to look at the other papers appearing in same issue as~\cite{Helmholtz} including works in Englisch, French, German, and Latin, as well as authors such as Boltzmann, Cayley, Frobenius, Fuchs, Hamburger, Hermite, Kronecker, Kummer, Lipschitz, Minkowski, Picard, Runge, Segre, Sylvester, Thom\'e, and Weingarten, who have many mathematical concepts named in their honor.}.

More generally, effectively all techniques in training machine learning algorithms rely on some form of optimization. In Section~\ref{cekk_sec:opti}, we have already seen how one can use stochastic dynamical systems to understand the convergence of stochastic gradient descent. Similar principles apply to almost all optimization algorithms as they usually leverage iterative dynamics, i.e., we define our algorithm as a time-dependent dynamical system that depends upon the function we want to optimize. This translation principle applies to all classical optimization methods, such as gradient descent, stochastic gradient descent, nonlinear programming techniques, Newton-type methods, etc., which overwhelmingly employ iteration techniques. Furthermore, many meta-heuristics just use more complex dynamical systems, such as simulated annealing, Metropolis-Hastings, consensus-based optimization, among many others. One may object that it may seem unnecessary to bring in a dynamical systems perspective since most algorithms ``work well in practice''. This argument does apply to a certain extent but this viewpoint then excludes our understanding, why and when optimization fails. Even for standard Newton-type methods, it is well-known that the dynamics of finding good starting values for the convergence of Newton's method can be extremely complex, even for extremely simple problems. In the context of machine learning, another good issue to illustrate this point are exploding or vanishing gradients. Suppose we have a loss function $\cL:\R^D\ra \R$ and consider the associated discrete-time or continuous-time gradient dynamics\index{Gradient Descent (GD)}\index{Gradient Flow}
\benn
p_{k+1} = p_k-\alpha\nabla (\cL_k)\qquad \text{or}\qquad p'=-\nabla \cL(p).
\eenn      
For didactic simplicity, let us just use continuous time and use the gradient flow ODE (the same line or argument holds in the discrete-time case). Let us even assume that there exists a single global minimum and that the optimization problem is convex, quadratic, and $D=2$. This yields a linear ODE system
\be
\label{cekk_eq:ODE2Dtrivial}
\frac{\txtd p}{\txtd t}=p'=Ap,\qquad p(0)=p_0,
\ee
where $p_0$ is our initial guess. In fact, we can even pick $A$ as a diagonal matrix $A=\textnormal{diag}(-1,-\varepsilon)$, where $0<\varepsilon\ll 1$ is a small parameter. It is trivial to see that $p(t)\ra 0$ as $t\ra 0$, i.e., we do converge to the global minimum $p=0$ of the loss function $\cL(p)=\frac12p_1^2+\frac12 \varepsilon p_2^2$, which is the equilibrium point of the ODE~\eqref{cekk_eq:ODE2Dtrivial}. Yet, since $\varepsilon$ is small, we only see fast convergence of the first component $p_1$. For $p_1$ small, we then get that $Ap$ is small as well, so convergence is extremely slow due to an almost vanishing gradient. More precisely, the ODE~\eqref{cekk_eq:ODE2Dtrivial} has a multiple time scale\index{Multi-Scale System} structure that causes almost vanishing gradients; note that in practical terms, due to numerical error and/or stochastic optimization schemes, we never have exactly vanishing gradients, but one struggles with the problem of gradients very close to zero. The same multiscale effect can just be much more pronounced in many, much more complicated problems. Indeed, a choice of architecture, activation function, layer depths for the feed-forward case, and many other factors can implicitly generate a multiscale structure. Even if the optimization problem is well-posed in principle, as in our trivial example, vanishing gradients might make it very difficult to compute in practice. Of course, a simple inversion of the parameter $\varepsilon$ to $1/\varepsilon$ gives extremely large gradients. In our example, this could even be achieved by just rescaling time and defining $s\coloneqq t\varepsilon$. Many classical problems in science and engineering naturally lead to multiscale structures, including e.g.~chemical reactions, biological neuron dynamics, laser systems, and evolutionary processes, just to name a few~\cite{Kuehn2015}. Therefore, it is not really surprising that many ML challenges, which can be rephrased as dynamical systems, also encounter this issue. In summary, it is quite natural to expect that vanishing/exploding gradients occur frequently in ML, which can be analyzed further using dynamics techniques, e.g., in the context of chaotic systems~\cite{Hessetal}.  

\section{Conclusion}\label{cekk_sec:concl}

Based upon our previous discussion, one may identify a general strategy for other (known or even unknown) topics arising in AI/ML. Any aspect purely linked to ``learning'' must be intrinsically connected to a dynamical viewpoint. Indeed, the process of training will often be a sequential or time-dependent one, refining the parametrization of a machine learning model. For example, we are constantly gathering new data, so we are typically going to have to train new or retrain existing models, which are both dynamical processes. Regarding the structural ``processing of the information'', e.g., via deep neural networks, there is always a sequential conversion step mapping inputs to outputs, which can be viewed as the result of a (finite-time) dynamical process. Combining both processes leads to non-autonomous dynamics, adaptive/co-evolutionary network dynamics, optimal control challenges, or variations thereof that can yet again be viewed from a dynamical systems standpoint. To understand the fundamental difficulty level of an AI/ML problem rigorously and to figure out which technical tools are available to solve it, we suggest that it can pay enormous dividends to extract the dynamical core issues first. For example, it is well-known why certain particle systems are much more difficult to study than others, why bifurcations in non-autonomous systems are structurally very different from autonomous ones, or why new effects can be induced by stochastic/random dynamical systems. 

Of course, there are challenges stemming from even more complex nonlinear systems in the context of AI that cannot be mathematically proven just by pencil-and-paper dynamics techniques. This is a well-known issue already for many quite low-dimensional classical dynamical systems~\cite{Tucker}. Yet, even in these cases there are rigorous methods available that can computationally validate an observed dynamical phenomenon. For example, within our SPP2298 project, that has led to various collaborations with other researchers focusing on computer-assisted proofs~\cite{KuehnQueirolo}, i.e., using a computer-assisted proof to prove that a neural network generates certain dynamical patterns. Although this direction will not be pursued further in this overview, it illustrates nicely that emerging directions in AI/ML are often coupled to techniques that arose due to challenges in dynamical systems. 

\bibliographystyle{plain}               
\bibliography{bibfileCEKK}               

\end{document}